\documentclass[a4paper]{siamltex}

\usepackage[english]{babel}
\usepackage[dvips]{graphicx}
\usepackage{amssymb}
\usepackage{amsmath}
\usepackage{xcolor}
\usepackage{url}
\usepackage{hyperref}

\newcommand{\phib}{\boldsymbol\phib}
\newcommand{\bm}[1]{\mathbf{#1}}

\newcommand{\cD}{\mathcal{D}}

\newcommand{\RR}{\mathbb{R}}
\newcommand{\dx}{\,dx}
\newtheorem{example}{Example}%
\newtheorem{remark}{Remark}%

\title{Averaged Nystr\"om interpolants for the solution of Fredholm
integral equations of the second kind}

\author{
L. Fermo\thanks{Department of Mathematics and Computer Science, University of Cagliari, 
Via Ospedale 72, 09124 Cagliari, Italy, email: fermo@unica.it, rodriguez@unica.it.}
\and
L. Reichel\thanks{Department of Mathematical Sciences, Kent State University, Kent, OH 
44242, USA, email: reichel@math.kent.edu.}
\and
G. Rodriguez\footnotemark[1]
\and 
M. M. Spalevi\'c\thanks{Department of Mathematics, Faculty of Mechanical Engineering, 
University of Belgrade, Kraljice Marije 16, 11120 Belgrade 35, Serbia, email: 
mspalevic@mas.bg.ac.rs.}
}

\begin{document}
\maketitle

\begin{abstract}
Fredholm integral equations of the second kind that are defined on a
finite or infinite interval arise in many applications. This paper discusses
Nystr\"om methods based on Gauss quadrature rules for the solution of such
integral equations. It is important to be able to estimate the error in the
computed solution, because this allows the choice of an appropriate number of
nodes in the Gauss quadrature rule used. This paper explores the application of
averaged and weighted averaged Gauss quadrature rules for this purpose, and
introduces new stability properties for them.

\end{abstract}



\section{Introduction}\label{sec:introduction}
Fredholm integral equations of the second kind,
\begin{equation}\label{Fredholm}
f(y)+\int_\cD k(x,y) f(x) d \mu(x)=g(y), \quad y \in \cD,
\end{equation}
where the kernel $k$ and right-hand side function $g$ are given, the function $f$ is to be
determined, and  $d \mu(x)$ is a nonnegative measure supported on a bounded or 
unbounded domain, 
arise in many applications including image restoration
(when applying Tikhonov regularization)~\cite{BucciniDonatelli}, conformal
mapping~\cite{Nasser}, frequency analysis~\cite{NRS}, and tomography~\cite{RamlauStadler}; see 
also Atkinson~\cite{A} and Kress~\cite{Kress} for discussions on further applications. The 
present paper considers equations of the form \eqref{Fredholm} when $\cD$ is a bounded or 
infinite interval on the real axis, and the integral operator  
\begin{equation*}
(Kf)(y)=\int_\cD k(x,y) f(x) d \mu(x)
\end{equation*}
is a compact map from $\mathcal{X}$ to $\mathcal{X}$, where $\mathcal{X}$ is a suitable 
weighted Banach space; see Section \ref{sec:formulae} for its definition.
Suitable assumptions on the kernel and on the measure can be made to guarantee compactness of the integral operator $K$. Detailed results on this topic are available in~\cite[Chapter 5]{junghanns2021}.

The Nystr\"om method is one of the most popular approaches to compute an
approximate solution of Fredholm integral equations of the second kind; see,
e.g., Atkinson~\cite{A} or Kress~\cite{Kress}. The method is easy to implement
and use: the integral in \eqref{Fredholm} is replaced by an interpolatory
quadrature rule $K_m$ with $m$ nodes $x_1<x_2<\cdots<x_m$ on the interval
$\cD$, and the equation 
\begin{equation}\label{I+Km}
(I+K_m)f_m = g,
\end{equation}
where $I$ is the identity operator and $f_m$ is the unknown interpolant,
is required to hold at the nodes $y=x_i$,
$i=1,2,\ldots,m$. This yields the linear system of equations 
\begin{equation}\label{sys}
\sum_{j=1}^m[\delta_{ij}+ c_j k(x_j,x_i)] a_j=g(x_i), \quad i=1,2,\dots,m,
\end{equation}
with a coefficient matrix of order $m$. Here the $c_j$ are coefficients of the chosen 
quadrature rule, $a_j=f_m(x_j)$, and $\delta_{ij}$ is the Kronecker $\delta$-function, 
i.e., $\delta_{ii}=1$ and $\delta_{ij}=0$ for $i\ne j$. Assume that the integral equation
\eqref{Fredholm} has a unique solution in $\mathcal{X}$. This is the case when the null 
space of the corresponding operator is trivial, that is when $\mathcal{N}(I+K)=\{0\}$. 
Then, when $m$ is sufficiently large, the matrix of the linear system of equations 
\eqref{sys} is nonsingular and its condition number can be bounded independently of $m$; 
see~\cite{A} for details.

Having computed the solution $[a_1,a_2,\ldots,a_m]^T\in\RR^m$ of the linear system of 
equations \eqref{sys}, the Nystr\"om interpolant 
\begin{equation}\label{eqfm}
f_m(y)=g(y)-\sum_{j=1}^m c_j k(x_j,y) a_j, \quad y \in \cD,
\end{equation}
provides an approximate solution $f_m(y)$ of \eqref{Fredholm} that can be evaluated at 
any $y\in\cD$.  The Nystr\"om interpolant \eqref{eqfm} is known to converge to the exact 
solution $f(y)$ of \eqref{Fredholm} with the same rate of convergence as the quadrature 
rule used; see, e.g.,~\cite{A} for details.

An important aspect of the Nystr\"om method is the choice of the quadrature formula. We would
like the quadrature rule to be convergent in the weighted function space ${\mathcal X}$
determined by the measure $d\mu(x)$. For this reason, Nystr\"om methods often are based on
Gauss quadrature rules, for which there exists a wide literature. Indeed, Gauss quadrature 
formulas associated with different measures have been applied to Nystr\"om methods, both 
for a single Fredholm integral equation of the second kind~\cite{FermoRusso,MastroianniMilovanovicNotarangelo,MastroianniNotarangelo,OccorsioRusso} 
and for systems of such equations~\cite{DeBonisLaurita,DeBonisMastroianni} in function 
spaces suited to handle possible pathologies of the solution of \eqref{Fredholm}.

In a sequence of papers Laurie~\cite{Laurie1} and Spalevi\'c with
collaborators~\cite{drst2016,ReichelSpalevic2021,Spalevic2007,Spalevic2017}
developed averaged and
weighted averaged Gaussian quadrature rules. These rules are convex combinations of two 
quadrature rules $G_m$ and $G^+_{m+1}$ with $m$ and $m+1$ nodes, respectively, where $G_m$
is an $m$-point Gauss rule and $G^+_{m+1}$ is an $(m+1)$-node quadrature rule that is 
related to $G_m$. Averaged and weighted averaged rules have been applied to estimate the 
quadrature error in Gauss rules. It is the purpose of the present paper to explore their 
application to estimate the error in Nystr\"om interpolants \eqref{eqfm}.

Since the averaged and weighted averaged rules have $2m+1$ nodes and weights, their 
straightforward application in a Nystr\"om method requires the solution of a linear system
of equations with a $(2m+1)\times(2m+1)$ matrix. We will explore the possibility of 
reducing the computational effort required when using these rules. 

The averaged rule introduced by Laurie~\cite{Laurie1} is the average of an
$m$-point Gauss rule and an $(m+1)$-point anti-Gauss rule. The application of
pairs of Gauss and anti-Gauss rules to the estimation of the error in Nystr\"om
interpolants has recently been described in~\cite{DFR2020}.

In the beginning of this paper, we analyze the weighted averaged rules
described in~\cite{Spalevic2007,Spalevic2017,ReichelSpalevic2022} and show
their stability and convergence in weighted function spaces. These results
extend those shown for anti-Gauss rules in~\cite{DFR2020} in that they involve
more general quadrature rules, weight functions, and domains. Moreover, our
results are shown under less restrictive assumptions than those
in~\cite{DFR2020}, and include the latter results. In the second part of the
paper, our discussion focuses on the use of averaged rules of
Laurie~\cite{Laurie1} and weighted averaged rules of
Spalevi\'c~\cite{Spalevic2007,Spalevic2017} to estimate the error in the
Nystr\"om
interpolant \eqref{eqfm}. Finally, new iterative methods are developed to solve
the linear systems of equations associated with Nystr\"om's method. These
methods exploit the structure of the coefficient matrix, and their convergence
is studied. To the best of our knowledge, this is the first time that such
methods are proposed.

This paper is organized as follows. Section \ref{sec:formulae} reviews averaged and 
weighted averaged Gaussian quadrature rules. New stability results are shown.
Illustrations of their performance for some 
classical measures are presented in Section \ref{sec3}. The application of averaged and
weighted averaged Gauss rules to the estimation of the error in Nystr\"om interpolants is 
discussed in Section \ref{sec:nystrom}, and several iterative methods for the computation
of Nystr\"om interpolants are described in Section \ref{itmeth}. Computed examples are 
presented in Section \ref{examples}, and concluding remarks can be found in Section 
\ref{sec5}.

\section{Averaged and weighted averaged Gauss quadrature rules}\label{sec:formulae}

Let
\begin{equation}\label{integral}
I(f) = \int_{\cD} f(x) d\mu(x)
\end{equation}
for some nonnegative measure $d\mu$ with infinitely many points of support, and let 
$\{p_k\}_{k=0}^\infty$ be the sequence of monic orthogonal polynomials associated with 
this measure, i.e., $p_k$ is a polynomial of degree $k$ with leading coefficient one such
that
\begin{equation}\label{orthonormality}
\langle x^j, p_k \rangle_\mu :=
\int_{\cD} x^j p_k(x) d\mu(x) = 0, \qquad j=0,1,\dots,k-1.
\end{equation} 
The above inner product defines a Hilbert space $L^2_\mu$ with induced norm
$\|f\|_\mu=\langle f,f\rangle_\mu^{1/2}$.

It is well known that the polynomials $p_k$ satisfy a recurrence relation of the form
\begin{equation}\label{recurrence}
\begin{cases}
p_{-1}(x)=0, \quad p_0(x)=1, \\
p_{k+1}(x)=(x-\alpha_k) p_{k}(x)-\beta_k p_{k-1}(x), \qquad k=0,1,2,\dots,
\end{cases}
\end{equation}
where the recursion coefficients are given by
\begin{equation*}
\begin{aligned}
\alpha_k &= \frac{\langle x p_k,p_k\rangle_\mu}{\langle p_k,p_k\rangle_\mu}, 
\qquad  &&k\geq 0, \\
\beta_k &= \frac{\langle p_k,p_k\rangle_\mu}{\langle p_{k-1},p_{k-1}\rangle_\mu},
\qquad  &&k\geq 1, \qquad  \beta_0 = \langle p_0,p_0 \rangle;
\end{aligned}
\end{equation*}
see, for instance,~\cite[Theorem 1.27]{Gautschi04}. 

The zeros of each polynomial $p_k$, $k\geq 1$, live in the convex hull of the support of
$d\mu$ and are distinct; see, e.g.,~\cite{Gautschi04}. Let 
$x_1^{(G)}<x_2^{(G)}<\cdots<x_m^{(G)}$ denote the zeros of $p_m$. They are known to be the
eigenvalues of the symmetric tridiagonal matrix 
\begin{equation*}
J_m=
\begin{bmatrix}
\alpha_0 & \sqrt{\beta_1} \\
\sqrt{\beta_1} & \alpha_1 & \sqrt{\beta_2} \\
& \sqrt{\beta_2} & \alpha_2 & \ddots \\ 
& & \ddots & \ddots & \sqrt{\beta_{m-1}} \\
& & & \sqrt{\beta_{m-1}} &  \alpha_{m-1} \\
\end{bmatrix};
\end{equation*} 
see~\cite{Gautschi04}. This matrix has orthogonal eigenvectors. Let $v_{k,1}$ denote the 
first component of a normalized real eigenvector associated with the eigenvalue
$x_k^{(G)}$. Then the $m$-point Gauss rule associated with the measure $d\mu$ is given by 
\begin{equation}\label{Gn}
G_m(f)= \sum_{k=1}^{m} \lambda_k^{(G)} f(x_k^{(G)}),
\end{equation} 
where the weights $\lambda_k^{(G)}$ (also known as Christoffel numbers) can be determined 
as $\lambda_k^{(G)}=\beta_0 v_{k,1}^2$; see~\cite{Gautschi04,GW}.  The rule \eqref{Gn} is
known to be exact for all polynomials in $\mathbb{P}_{2m-1}$; see, e.g.,
\cite{Gautschi04}. Here and throughout this paper $\mathbb{P}_k$ denotes the
set of polynomials of degree at most $k$.

Consider the extended symmetric tridiagonal matrix
$$
\widetilde{J}_{m+1}=
\begin{bmatrix}
J_m & \sqrt{2\beta_m}\bm{e}_m   \\
\sqrt{2\beta_m} \bm{e}^T_m & \alpha_m \\
\end{bmatrix},
$$
where $\bm{e}_k=[0,\dots,0,1,0,\ldots,0]^T$ stands for the $k$th 
axis vector of suitable dimension. Let $\{\widetilde{x}_k\}_{k=1}^{m+1}$ denote the
eigenvalues of $\widetilde{J}_{m+1}$. They are real and distinct, but they are not guaranteed 
to live in the convex hull of the support of the measure $d\mu$; see, e.g.,~\cite{Laurie1}
for a discussion. Let $\widetilde{v}_{k,1}$ denote the first component of a normalized 
real eigenvector associated with the eigenvalue $\widetilde{x}_k$. Then the 
$\widetilde{x}_k$ are the nodes and the 
$\widetilde{\lambda}_k=\beta_0\widetilde{v}_{k,1}^2$ are the weights of the 
$(m+1)$-point anti-Gauss rule
\begin{equation}\label{Gm+1}
\widetilde{G}_{m+1}(f) = \sum_{k=1}^{m+1} \widetilde{\lambda}_k f(\widetilde{x}_k)
\end{equation}
introduced by Laurie~\cite{Laurie1}. This rule satisfies
\begin{equation}\label{oppsign}
(I-\widetilde{G}_{m+1})(p)=-(I-G_m)(p),\qquad\forall~p\in\mathbb{P}_{2m+1};
\end{equation}
see~\cite{Laurie1} for details. The degree of exactness of the anti-Gauss rule 
$\widetilde{G}_{m+1}$ is at least $2m-1$. This follows from \eqref{oppsign}.

Laurie~\cite{Laurie1} also introduced the averaged rule
\begin{equation}\label{averaged1}
\widetilde{A}_{2m+1}(f) := \frac{1}{2}\left(G_m(f)+\widetilde{G}_{m+1}(f) \right),
\end{equation}
which has the following properties:
\begin{enumerate}
\item 
It has $2m+1$ nodes. Its evaluation requires the calculation of the integrand $f$ at 
$2m+1$ nodes. However, $m$ of these function values also are required to evaluate the 
Gauss rule $G_m(f)$. Therefore, the additional computational effort demanded when
calculating the averaged rule \eqref{averaged1} is only $m+1$ function evaluations.
\item 
All its $2m+1$ weights are positive.
\item 
Its degree of exactness is at least $2m+1$, i.e., $\widetilde{A}_{2m+1}(f)=I(f)$ for all
$f\in\mathbb{P}_{2m+1}$. For some measures, the averaged rule agrees with the
Gauss--Kronrod rule and its degree of exactness is higher; see,
e.g.,~\cite{drst2016,Notaris,Notaris2019,Spalevic2020}.
\item 
For certain measures, the averaged rule \eqref{averaged1} is internal, i.e., all nodes
live in the convex hull of the support of the measure $d\mu$. 
For instance, when $d \mu(x)= x^\alpha e^{-x} dx$ and $\cD=\mathbb{R}^+$, the rule
\eqref{averaged1} is internal for $\alpha>-1$. If instead 
$d\mu(x)=(1-x)^\alpha(1+x)^\beta dx$ and $\cD=[-1,1]$, then the averaged rule is internal 
only for suitable values of $\alpha$ and $\beta$; see ~\cite{Laurie1}.
\item 
The averaged rule \eqref{averaged1} furnishes an estimate for the error $I(f)-G_m(f)$ since
$$
I(f)-G_m(f)\approx\widetilde{A}_{2m+1}(f)-G_m(f)=
\frac{1}{2}(\widetilde{G}_{m+1}(f)-G_m(f)). 
$$
\end{enumerate}

Spalevi\'c~\cite{Spalevic2007} constructed a symmetric tridiagonal matrix of order $2m+1$
whose eigenvalues are the nodes of the averaged rule \eqref{averaged1} and the weights can
be computed from the first component of the associated normalized real eigenvectors. This 
construction lead Spalevi\'c to the definition of the weighted averaged quadrature rule 
\begin{equation}\label{G2m+1}
\widehat{A}_{2m+1}(f)=\sum_{k=1}^{2m+1}\widehat{\lambda}_k f(\widehat{x}_k)
\end{equation}
associated with the Gauss rule $G_m$. The nodes of the rule \eqref{G2m+1} are the 
eigenvalues of the symmetric tridiagonal matrix
\begin{equation}\label{breveJ}
\widehat{J}_{2m+1}= 
\begin{bmatrix}
J_m & \sqrt{ \beta_m}\bm{e}_m & 0 \\
\sqrt{ \beta_m} \bm{e}^T_m & \alpha_m & \sqrt{ \beta_{m+1}}\bm{e}^T_1 \\
0 & \sqrt{ \beta_{m+1}} \bm{e}_1 & Z_m J_m Z_m
\end{bmatrix}\in\RR^{(2m+1)\times(2m+1)},
\end{equation}
where $Z_m\in\RR^{m\times m}$ is the row-reversed identity matrix. The weights of the
quadrature rule \eqref{G2m+1} are the square of the first component of normalized real
eigenvectors of \eqref{breveJ} multiplied by $\beta_0$. Spalevi\'c~\cite{Spalevic2007} 
showed the following properties of the quadrature formula \eqref{G2m+1}:
\begin{enumerate}
\item[(A)] 
The formula requires $2m+1$ evaluations of the integrand $f$. When the nodes $x_j^{(G)}$
of the Gauss rule \eqref{Gn} and the nodes $\widehat{x}_j$ of the weighted averaged rule 
\eqref{G2m+1} are ordered in increasing order, the quadrature nodes satisfy 
$$\widehat{x}_{2j}=x_j^{(G)}, \qquad j=1,2,\dots,m. $$     
\item[(B)] 
All the weights $\widehat{\lambda}_k$ are positive. 
\item[(C)] 
The quadrature rule is exact for polynomials of degree at least $2m+2$. If the measure 
$d \mu(x)$ is symmetric with respect to the origin, then the degree of exactness is 
at least $2m+3$. For certain measures $d\mu$, the degree of exactness is much
higher; see~\cite{drst2016,Notaris2019,Spalevic2020}.
\item[(D)] 
If $d\mu(x)= x^\alpha e^{-x} dx$ and $\cD=\mathbb{R}^+$, then the formula is internal 
when $\alpha>1$. If $d\mu(x)=(1-x)^\alpha (1+x)^\beta dx$ and $\cD=[-1,1]$, then it is
internal only for certain values of $\alpha$ and $\beta$.
\item[(E)] 
The quadrature rule suggests the error estimate
\begin{equation}\label{errorGm}
I(f)-G_m(f)\approx\widehat{A}_{2m+1}(f)-G_m(f).
\end{equation}
Computed examples reported in~\cite{ReichelSpalevic2022} show the weighted averaged
quadrature rule \eqref{G2m+1} for many integrands to give higher accuracy than suggested
by its degree of exactness. This also holds to a lesser extent for the averaged rule
\eqref{averaged1}. This property of the rule \eqref{G2m+1} results in that the right-hand 
side of \eqref{errorGm} for many integrands provides an accurate estimate of the left-hand 
side. The present paper uses this property to determine accurate estimates of the error in
computed approximate solutions of Fredholm integral equations of the second kind.
\end{enumerate}

When applying the quadrature rule \eqref{G2m+1}, one generally also evaluates the Gauss
rule \eqref{Gn}. Therefore, the nodes and weights of the representation 
\begin{equation}\label{G2m+1_bis}
\widehat{A}_{2m+1}(f)=\frac{\beta_{m+1}}{\beta_m+\beta_{m+1}} G_m(f)+
\frac{\beta_{m}}{\beta_m+\beta_{m+1}} G^*_{m+1}(f)
\end{equation}
can be evaluated faster than computing the eigenvalues and first components of normalized
eigenvectors of the matrix \eqref{breveJ}. Here $G^*_{m+1}$ is the quadrature rule 
determined by the matrix
\begin{equation}\label{Jstar}
{J}^*_{m+1}=
\begin{bmatrix}
J_m & \sqrt{\beta_m+\beta_{m+1}}\bm{e}_m \\
\sqrt{\beta_m+\beta_{m+1}} \bm{e}^T_m & \alpha_m 
\end{bmatrix};
\end{equation}
that is
\begin{equation}\label{Gstar}
\int_{\cD} f(x) d \mu(x) = \sum_{k=1}^{m+1} \lambda_k^* f(x^*_k)+e^*_{m+1}(f)=:
G^*_{m+1}(f)+e^*_{m+1}(f),
\end{equation}
where $e^*_{m+1}(f)$ denotes the quadrature error in $G^*_{m+1}(f)$;
see~\cite{ReichelSpalevic2021} for a derivation of \eqref{G2m+1_bis} and a
discussion on the computational effort.

The representation \eqref{G2m+1_bis} shows that the rule
$\widehat{A}_{2m+1}(f)$ is a weighted average of the Gauss rule $G_m(f)$ and
the quadrature rule $G^*_{m+1}(f)$. 
It also shows that one can evaluate the error estimate \eqref{errorGm} as 
\[
\frac{\beta_m}{\beta_{m}+\beta_{m+1}} \left(G^*_{m+1}(f)-G_m(f)\right).
\]
The following expression for the weights of the rule $G^*_{m+1}$ is believed to be 
new.

\begin{theorem}\label{prop1}
The degree of exactness of the quadrature rule $G_{m+1}^*$ in \eqref{G2m+1_bis}
is at least $2m-1$, the nodes $x^*_k$ interlace with the Gauss nodes $x_k$,
and the weights are given by
$$
\lambda_k^* = \frac{\beta_m+\beta_{m+1}}{\beta_m} 
\frac{\|p_m\|_\mu^2}{q'_{2m+1}(x^*_k)}>0,
$$
where 
$$
q_{2m+1}(x)=p_m(x) p^*_{m+1}(x) 
\quad\text{and}\quad
p^*_{m+1}(x) = \prod_{k=1}^{m+1}(x-x^*_k).
$$
\end{theorem}

\begin{proof}
The first part of the thesis follows from the representation \eqref{G2m+1_bis} of $\widehat{A}_{2m+1}$ and
the fact that the degree of exactness of both the weighted averaged rule 
$\widehat{A}_{2m+1}$ and the Gauss rule $G_{m}$ is at least $2m-1$. 

Let us consider the function $f_j(x)=q_{2m+1}(x)/(x-x^*_j)$.
Since $G_m(f_j)=0$ for any $j$, eq. \eqref{G2m+1_bis} yields
$$
\begin{aligned}
\int_{\cD} f_j(x) d \mu(x) &= \frac{\beta_m}{\beta_m+\beta_{m+1}}
\sum_{k=1}^{m+1} \lambda^*_k \, \frac{q_{2m+1}(x^*_k)}{(x^*_k-x^*_j)} \\
& = \frac{\beta_m}{\beta_m+\beta_{m+1}} \lambda^*_j q_{2m+1}'(x^*_j).
\end{aligned}
$$
The last equality is a consequence of the well-known representation for the
Lagrange polynomial associated with the quadrature node $x_j^*$ of
formula \eqref{G2m+1_bis}, 
$$
L_j(x) = \frac{q_{2m+1}(x)}{(x-x_j^*)q_{2m+1}'(x_j^*)},
$$
where we note that $q_{2m+1}'(x^*_j)=p_m(x^*_j) (p^*_{m+1})'(x^*_j)$.

On the other hand, for any $j$ there exists a polynomial $q_{m-1,j}$ of degree
$m-1$ such that
\begin{align*}
\int_{\cD} f_j(x) d\mu(x) &=  \int_{\cD} p_m(x) [x^m+q_{m-1,j}(x)] d\mu(x) 
=\|p_m\|_\mu^2.
\end{align*}
Combining the above two equalities we obtain the expression given for
$\lambda^*_k$, whose positivity follows from its definition in terms of the
squared first component of a normalized real eigenvector.
Finally, the interlacing property of the nodes follows by applying Cauchy's
interlacing result to the matrix \eqref{Jstar}; see, e.g.,~\cite[Theorem 3.3]{GM}.  
\end{proof}

The next lemma, which will be useful in the sequel, gives a Markov--Stieltjes-type 
inequality. An analogous inequality is well known for classical Gauss rules (see, for
instance,~\cite{Freud}), but the inequality has never been proved for averaged rules.  

\begin{lemma}
For a fixed quadrature node $x^*_k$ of $G_{m+1}^*$, we have the bounds 
\begin{align} \label{bound1}
\sum_{i=1}^{k-1} \lambda^*_i 
\leq \int_{-\infty}^{x^*_k} d\mu(x) \leq 
\sum_{i=1}^{k} \lambda^*_i.
\end{align}
\end{lemma}

\begin{proof}
Let $P^{(k)}_{2m}$ and $Q^{(k)}_{2m}$ be two polynomials of degree $2m$ such that
$$P^{(k)}_{2m}(x^*_i)=\begin{cases} 1,& 1\leq i<k,\\0,& k\leq i\leq m+1,
\end{cases}\quad\mbox{and}\quad
\frac{dP^{(k)}_{2m}}{dx}(x^*_i)=0, \quad \forall i \neq k$$
and
$$Q^{(k)}_{2m}(x^*_i)=\begin{cases} 1,& 1\leq i \leq k,\\0,& k<i\leq m+1,
\end{cases}\quad\mbox{and}\quad
\frac{dQ^{(k)}_{2m}}{dx}(x^*_i)=0, \quad \forall i \neq k.$$
These polynomials are uniquely determined; see, for
instance,~\cite[Lemma~1.3]{Freud}. Following the proof
of~\cite[Theorem~5.2]{Freud}
one obtains that for each real $x$, we have 
\begin{equation}\label{heaviside}
P^{(k)}_{2m}(x) \leq H_k(x) \leq Q^{(k)}_{2m}(x),
\end{equation} 
where $H_k$ is the shifted Heaviside function defined by
$$H_k(x)=\begin{cases} 1,&x \leq x^*_k,\\0,&x> x^*_k.
\end{cases}$$
Then, since by  \eqref{G2m+1_bis}
one has
$$
G^*_{m+1}(q) = I(q) + \frac{\beta_{m+1}}{\beta_{m}} \left( I(q)
- G_m(q) \right), \qquad
\forall q \in \mathbb{P}_{2m+2},
$$
we can write
\begin{align*}
0< \sum_{i=1}^{k-1} \lambda^*_i 
&= \sum_{i=1}^{m+1} \lambda^*_i P^{(k)}_{2m}(x^*_i)
\leq I(H_k) + \frac{\beta_{m+1}}{\beta_m} \left( I(P^{(k)}_{2m}) -
	G_m(P^{(k)}_{2m}) \right) \\ 
&= \int_{-\infty}^{x^*_k} d\mu(x) + \frac{\beta_{m+1}}{\beta_m}
	e_m(P^{(k)}_{2m})
\leq \int_{-\infty}^{x^*_k} d\mu(x),
\end{align*}
where $e_m(f)$ represents the quadrature error for the Gauss rule.
To justify the last inequality, let us recall the error representation of the
Gauss rule~\cite[Formula~(1.4.14)]{Gautschi04}
$$
e_m(f) = \frac{f^{(2m)}(\xi)}{(2m)!} \|p_m\|_\mu^2, \qquad 
\xi\in\cD\setminus\partial\cD.
$$
By virtue of \eqref{heaviside}, as $H_k(x)$ is a piecewise constant function and the 
polynomial $P_{2m}^{(k)}$ has even degree, its leading coefficient has to be negative to 
satisfy the inequality for $x\to\pm\infty$. Therefore, $e_m(P^{(k)}_{2m})<0$.

Similarly, we have
\begin{align*}
0< \sum_{i=1}^{k} \lambda^*_i 
&= \sum_{i=1}^{m+1} \lambda^*_i Q^{(k)}_{2m}(x^*_i)
\geq I(H_k) + \frac{\beta_{m+1}}{\beta_m} \left( I(Q^{(k)}_{2m}) -
	G_m(Q^{(k)}_{2m}) \right) \\ 
&= \int_{-\infty}^{x^*_k} d\mu(x) + \frac{\beta_{m+1}}{\beta_m}
	e_m(Q^{(k)}_{2m})
\geq \int_{-\infty}^{x^*_k} d\mu(x),
\end{align*}
since 
the polynomial $Q_{2m}^{(k)}$ of degree $2m$ has
a positive leading coefficient and, consequently, $e_m(Q^{(k)}_{2m})>0$.
Thus, \eqref{bound1} follows.
\end{proof}

\smallskip
Let us now investigate the stability and convergence of the formula
$G^*_{m+1}(f)$ in different function spaces. 
In the set of all continuous functions  $C(\cD)$ equipped with the uniform norm
\[
\displaystyle \|f\|_\infty =\sup_{x \in \cD} \vert f(x) \vert,
\]
the stability is an immediate consequence of the equality
$$\|G^*_{m+1}\| := \sup_{\|f\|_\infty=1} \vert G^*_{m+1}(f) \vert 
=\sum_{k=1}^{m+1} \lambda^*_k =\int_{\cD} d \mu(x)<\infty,$$
since this implies that $\sup_m\|G^*_{m+1}\|<\infty$.

The quadrature error tends to zero as fast as the best approximation error by
polynomials of degree less than $2m-1$, 
$$E_{2m-1}(f):=\inf \{ \|f-P\|_\infty \, : \, P \in \mathbb{P}_{2m-1} \},$$
since for the quadrature error $e^*_{m+1}$ of $G^*_{m+1}$ it holds
$$
\begin{aligned}
 \vert e^*_{m+1}(f) \vert = \vert e^*_{m+1}(f-P) \vert  \leq \mathcal{C} E_{2m-1}(f), \qquad \mathcal{C}\neq \mathcal{C}(m,f). 
\end{aligned}
$$
Here and in the sequel, $\mathcal{C}$ is a positive constant which may assume different 
values in different formulas. We write $\mathcal{C} \neq \mathcal{C}(a,b,\dots)$ to 
indicate that $\mathcal{C}$ is independent of the parameters $a,b,\dots$.

Introduce a bounded weight function $u: \cD \to \mathbb{\RR}$ that is positive on the 
support of $d\mu$, and satisfies
\begin{equation}\label{pesou}
\int_\cD  \frac{d\mu(x)}{u(x)}<\infty, \qquad \int_\cD x^k \, u(x) d\mu(x)
<\infty, \qquad k=0,1,\ldots.
\end{equation}
We define the weighted space $C_u(\cD)$ as the set of all continuous function
$f \in C(\cD\setminus\partial\cD)$ such that  $fu \in C(\cD)$, equipped with the
weighted uniform norm $\|fu\|_\infty$.
For smoother functions, we consider Sobolev-type spaces of index $r \geq 1$,
\begin{equation}\label{Sobolev}
W^r_u(\cD) = \{f \in C_u(\cD) \, \vert  \, f^{(r-1)} \in AC(\cD\setminus\partial\cD) \text{ and } 
\|f^{(r)}\varphi^r u\|_\infty < \infty \},
\end{equation}
where $AC(\cD\setminus\partial\cD)$ denotes the set of absolutely
continuous functions on $\cD\setminus\partial\cD$ and 
\begin{align}\label{phix}
\varphi(x)= \begin{cases}
\sqrt{(b-x)(x-a)}, & \textrm{if} \quad  \cD=[a,b], \\
\sqrt{x-a}, & \textrm{if} \quad \cD=[a,\infty), \\
1, & \textrm{if} \quad  \cD=\mathbb{R}.
\end{cases}
\end{align}

The following result shows the stability of the quadrature
rule $G^*_{m+1}$ in $C_u(\cD)$.

\begin{theorem} \label{stability}
Let  $u$ be a bounded weight function that is positive on the support of $d\mu$ and 
satisfies \eqref{pesou}. Assume that
$$
u(x^*_k) \neq 0,  \qquad k=1,2,\dots,m+1.
$$
Then the formula $G^*_{m+1}$ is stable, i.e.,
$$ \sup_m \left( \sup_{\|fu\|_\infty=1} \vert G^*_{m+1}(f) \vert  \right)
=\sup_m \left( \sum_{k=1}^{m+1} \frac{\lambda^*_k}{u(x^*_k)} \right)
<\infty.$$
\end{theorem} 

\begin{proof}
The bounds \eqref{bound1} imply
\begin{equation}\label{lambdak}
\lambda^*_k = \sum_{i=1}^{k} \lambda^*_i - \sum_{i=1}^{k-1} \lambda^*_i 
\leq \int_{x^*_{k-1}}^{x^*_{k+1}} d \mu(x), \qquad
k=2,3,\ldots,m.
\end{equation}
By applying the left-hand side inequality in \eqref{bound1} for $k=2$, and the
right-hand side inequality for $k=m$, 
we can extend \eqref{lambdak} to $k=1$ and $k=m+1$ by defining
$x_0^*=-\infty$ and $x_{m+2}^*=\infty$.
It follows from the assumptions on $u$ that, for $k=1,\ldots,m+1$, there
exist uniformly bounded constants $\mathcal{C}_k$ such that
\begin{equation*}
\frac{1}{u(x^*_k)} \int_{x^*_{k-1}}^{x^*_{k+1}} d\mu(x)
= \mathcal{C}_k \int_{x^*_{k-1}}^{x^*_{k+1}} \frac{d\mu(x)}{u(x)}.
\end{equation*}
Setting $\mathcal{C}=\max_k\mathcal{C}_k$, we obtain from \eqref{lambdak} that
\begin{equation}\label{boundsum}
\begin{aligned}
\sum_{k=1}^{m+1} \frac{\lambda^*_k}{u(x^*_k)}  
& \leq \mathcal{C} \sum_{k=1}^{m+1} \int_{x^*_{k-1}}^{x^*_{k+1}} 
	\frac{d\mu(x)}{u(x)} \\
& = \mathcal{C} \int_{-\infty}^{x^*_{m+1}} \frac{d\mu(x)}{u(x)}
	+ \mathcal{C} \int_{x^*_1}^{\infty} \frac{d\mu(x)}{u(x)}
\leq 2 \mathcal{C} \int_{\cD} \frac{d\mu(x)}{u(x)} <\infty,
\end{aligned}
\end{equation}
as the measure $\mu(x)$ is supported on $\cD$.
This shows the stability of the formula.
\end{proof}
\medskip

The bound \eqref{boundsum} allows us to show convergence of the quadrature
rules $G^*_{m+1}$ as $m$ increases, i.e., 
$$
\displaystyle \lim_{m \to \infty} e^*_{m+1}(f) =0,
$$
and that $e^*_{m+1}(f)$ goes to zero as fast as the error of the best polynomial
approximation,
$$
E_{2m-1}(f)_u:=\inf \{ \|(f-P)u\|_\infty \, : \, P \in \mathbb{P}_{2m-1} \}.
$$
This is shown in the following corollary.

\begin{corollary}\label{corollary}
For each $f \in C_u(\cD)$ one has
\begin{equation}\label{ineq2.4}
\lvert e^*_{m+1}(f) \rvert \leq \mathcal{C} E_{2m-1}(f)_u, 
\end{equation}
where $\mathcal{C}\neq \mathcal{C}(m,f)$.
\end{corollary}

\begin{proof}
The inequality \eqref{ineq2.4} can be shown following, mutatis mutandis,
the proof of Theorem 5.1.7 in~\cite{MMlibro}. We report the main steps of 
the proof for the convenience of the reader. By the exactness of formula 
$G^*_{m+1}$ and by \eqref{boundsum}, we have for each polynomial 
$P \in \mathbb{P}_{2m-1}$ that
\begin{equation*}
\begin{aligned}
 \vert e^*_{m+1}(f) \vert  &=  \vert e^*_{m+1}(f-P) \vert  \\
&= \left\lvert  \int_{\cD} [f(x)-P(x)] d\mu(x) -\sum_{k=1}^{m+1} \lambda^*_k
[f(x^*_k)-P(x^*_k)] \right\rvert  \\
& \leq \|[f-P]u\|_\infty \left[ \int_{\cD} \frac{d\mu(x)}{u(x)}+\sum_{k=1}^{m+1} \frac{\lambda^*_k}{u(x^*_k)} \right] \\
& \leq \mathcal{C}' \, \|[f-P]u\|_\infty   \int_{\cD}  \frac{d\mu(x)}{u(x)} 
\end{aligned} 
\end{equation*} 
for some constant $\mathcal{C}'$ related to the constant $\mathcal{C}$ in 
\eqref{boundsum}. Taking the infimum with respect to $P$, we obtain the assertion.
\end{proof}

\begin{example}\rm
Let the function $f$ belong to the function space $C_u([-1,1])$ with
\begin{equation}\label{weightu}
u(x)=(1-x)^\gamma (1+x)^\delta, \qquad \gamma,\delta\geq 0,
\end{equation}
and assume that we would like to evaluate \eqref{integral}, where the measure
$d\mu(x)=w(x)dx$ is determined by the Jacobi weight function 
\begin{equation*}
w(x) =(1-x)^\alpha (1+x)^\beta, \qquad \alpha,\beta>-1.
\end{equation*}
Then the first inequality in \eqref{pesou} is satisfied if 
\begin{equation}\label{abcond}
\gamma<\alpha+1, \qquad  \delta<\beta+1,
\end{equation}
and Theorem \ref{stability} holds true. 

Now, let instead $f$ be a smoother function, namely, let $f$ belong to the
Sobolev space $W^r_u([-1,1])$ of index $r \geq 1$, 
defined in \eqref{Sobolev}.
Then~\cite[p.~172]{MMlibro}, 
$$
E_m(f)_u \leq \mathcal{C} m^{-r} \|f^{(r)} \varphi^ru\|_\infty,
$$
where $\varphi$ is given in \eqref{phix}. Corollary~\ref{corollary} then yields
$$ \vert e^*_{m+1}(f) \vert  \leq \frac{C}{m^r} \|f^{(r)} \varphi^ru\|_\infty.$$
\end{example}

\begin{example}\rm
Let the function $f$ belong to the function space $C_u([0,\infty))$ with
$$
u(x)=x^\gamma (1+x)^\delta e^{-x}, \qquad \gamma,\delta\geq 0,
$$
and assume that we would like to evaluate \eqref{integral}, where the measure
$d\mu(x)=w(x)dx$ is determined by the Laguerre weight function 
\begin{equation*}
w(x) =x^\alpha e^{-x}, \qquad \alpha>-1.
\end{equation*}
Then the first inequality in \eqref{pesou} is satisfied if $\alpha-\gamma>-1$
and $\delta>1+\alpha-\gamma$, and Theorem \ref{stability} holds true.

Additionally, let us consider the weighted Sobolev space $W^r_u([0,\infty))$ of
index $r \geq 1$; see~\eqref{Sobolev}--\eqref{phix}.
Then,~\cite[p.\,177]{MMlibro} yields
$$ \vert e^*_{m+1}(f) \vert  \leq \frac{C}{m^{r/2}} \|f^{(r)} \varphi^ru\|_\infty.$$
\end{example}

We conclude this section by showing the stability and convergence of the
quadrature rule $\widehat{A}_{2m+1}$.

\begin{corollary}\label{stabAhat}
Under the assumptions of Theorem~\ref{stability}, formula $\widehat{A}_{2m+1}$ is stable, i.e.,
$$
\sup_m \left( \sup_{\|fu\|_\infty=1} \vert \widehat{A}_{2m+1}(f) \vert  \right)
=\sup_m \left( \sum_{k=1}^{2m+1} \frac{\widehat{\lambda}_k}{u(\widehat{x}_k)}
\right) <\infty,
$$
and convergent, i.e.,
$$
\left\lvert I(f)-\widehat{A}_{2m+1}(f)\right\rvert \leq \mathcal{C} E_{2m-1}(f)_u,
\qquad \mathcal{C}\neq \mathcal{C}(m,f).
$$ 
\end{corollary}

\begin{proof}
Stability follows from the representation \eqref{G2m+1_bis}, the stability of the Gauss 
rule $G_m$, and Theorem~\ref{stability}. Convergence can be shown similarly as in
the proof of Corollary~\ref{corollary}.
\end{proof}

\section{Properties and performance of the considered quadrature rules}\label{sec3}
We analyze the quadrature rules \eqref{Gm+1}, \eqref{Gstar}, \eqref{averaged1}, and 
\eqref{G2m+1} for a few measures $d\mu(x)=w(x)dx$ with classical weight functions $w(x)$ 
that commonly arise in Fredholm integral equations of the second kind.

\subsection{Jacobi weight functions}
Let us consider polynomials \eqref{recurrence} that are orthogonal with respect
to the Jacobi weight function 
\begin{equation}\label{jw}
w(x)=(1-x)^\alpha (1+x)^\beta,
\end{equation}
for parameters $\alpha,\beta>-1$. The recursion coefficients $\alpha_k$ and $\beta_k$
in \eqref{recurrence} are explicitly known and can be expressed in terms of 
$\alpha$ and $\beta$ as follows:
\begin{equation*}
\begin{aligned}
\alpha_k&= \frac{\beta^2-\alpha^2}{(2k+\alpha+\beta)(2k+\alpha+\beta+2)}, 
& \quad k \geq 0,   \\
\beta_0 &= \frac{2^{\alpha+\beta+1} \Gamma(\alpha+1) \Gamma(\beta+1)}{\Gamma(\alpha+\beta+2)},   \\
\beta_k&= \frac{4k(k+\alpha)(k+\beta)(k+\alpha+\beta)}{(2k+\alpha+\beta)^2
((2k+\alpha+\beta)^2-1)}, & \quad k \geq 1,  
\end{aligned}
\end{equation*}
where $\Gamma(\cdot)$ denotes the Gamma function. If $\alpha^2=\beta^2$, then $\alpha_k=0$
for all $k\geq 0$. Moreover, when $\alpha,\beta\in\{-\frac{1}{2},\frac{1}{2}\}$, the
associated orthogonal polynomials are Chebychev polynomials of the first,
second, third, or fourth kinds, and $\beta_k=\frac{1}{4}$ for $k \geq 1$.
Since in this case $G^*_{m+1}\equiv\widetilde{G}_{m+1}$, we obtain from \eqref{G2m+1_bis} 
that
$$
\widehat{A}_{2m+1}
=\frac{1}{2}(G_m(f)+\widetilde{G}_{m+1}(f)), \qquad m\geq 2,
$$
i.e., the weighted averaged quadrature formula \eqref{G2m+1_bis} coincides with
the averaged formula \eqref{averaged1}.


For general $\alpha,\beta>-1$ in \eqref{jw}, we have 
\begin{equation}\label{betalim}
\lim_{m \to \infty} \beta_m=\frac{1}{4},
\end{equation}
and it follows that the coefficients for $G_m$ and $G_{m+1}^*$ in
\eqref{G2m+1_bis} tend to $\frac{1}{2}$ as $m$ increases, so that
$$
\lim_{m \to \infty} \left(\widehat{A}_{2m+1}(f) - \widetilde{A}_{2m+1}(f) \right) = 0.
$$
This implies that the quadrature rules $\widehat{A}_{2m+1}(f)$ and
$\widetilde{A}_{2m+1}(f)$ may produce significantly different results only for
small values of $m$.

\begin{example}\rm
Consider the integral
$$
\begin{aligned}
I_1& =\int_{-1}^1 x e^x \cos{(x+1)} \,dx=\dfrac{1+e^2 \cos{2}}{2e},\\
\end{aligned}
$$
The integral $I_1$ can be computed analytically. To illustrate the performance of the
quadrature rules without influence of round-off errors introduced during the computations, 
we carry out all computations of this section in high-precision arithmetic. Results 
determined in standard  double precision arithmetic are very close to those reported.

Table~\ref{table:I1} displays, for the integral $I_1$ and several small
values of $m$, the quadrature errors obtained by the Gauss rule $G_m$, the
anti-Gauss formula $\widetilde{G}_{m+1}$, the rule $G^*_{m+1}$, the averaged
formula $\widetilde{A}_{2m+1}$, and the weighted averaged rule
$\widehat{A}_{2m+1}$. The weighted averaged rule $\widehat{A}_{2m+1}$ can 
be seen to produce a more accurate approximation of $I_1$ than
$\widetilde{A}_{2m+1}$ for all values of $m$.
It also can be observed that the anti-Gauss rule $\widetilde{G}_{m+1}$ and
the rule $G^*_{m+1}$ give quadrature errors of opposite sign to that of the 
corresponding Gauss rule $G_m$.

\begin{table}[ht]
\begin{center}
\caption{Quadrature errors for the integral $I_1$.}
\label{table:I1}
\begin{tabular}{c|ccc|cc}
\hline
$m$ & $I_1-G_m$ & $I_1-\widetilde{G}_{m+1}$ & $I_1-G^*_{m+1}$ 
& $I_1-\widetilde{A}_{2m+1}$ & $I_1-\widehat{A}_{2m+1}$ \\
\hline
2 & $-7.93e-02$ &  $\phantom{-}7.93e-02$ &  $\phantom{-}7.65e-02$ & $-3.24e-05$ & 
$-7.88e-06$ \\ 
3 & $\phantom{-}6.29e-04$ & $-6.30e-04$ & $-6.21e-04$ & $-3.10e-07$ & 
$\phantom{-}3.00e-09$ \\ 
4 &  $\phantom{-}2.51e-05$ & $-2.51e-05$ & $-2.49e-05$ &  $\phantom{-}2.95e-10$ &  
$\phantom{-}1.73e-11$ \\ 
5 & $-4.77e-08$ & $\phantom{-}4.77e-08$ &  $\phantom{-}4.76e-08$ &  $\phantom{-}2.49e-12$
& $-7.36e-15$ \\ 
6 & $-8.10e-10$ &  $\phantom{-}8.10e-10$ &  $\phantom{-}8.08e-10$ & $-1.29e-15$ & 
$-3.84e-17$ \\ 
\hline
\end{tabular}
\end{center}
\end{table}


In Table \ref{table:I1}, as well as in the remainder of this section, the rule 
$\widehat{A}_{2m+1}$ was computed according to \eqref{G2m+1_bis}.


The rules $\widetilde{A}_{2m+1}$ and $\widehat{A}_{2m+1}$ can be used to 
estimate the quadrature error $(I-G_m)(f)$. A comparison of Table~\ref{table:I12} with 
the second columns of Table~\ref{table:I1} shows these error estimates to be quite accurate.

\begin{table}[ht]
\begin{center}
\caption{Quadrature error estimates for $G_m$ obtained by the averaged rules
for the integral $I_1$.}
\label{table:I12}
\begin{tabular}{ccc}
\hline
$m$ & $\widetilde{A}_{2m+1}-G_m$ & $\widehat{A}_{2m+1}-G_m$ \\
\hline
2 & $-7.93e-02$ & $-7.93e-02$ \\ 
3 & $\phantom{-}6.29e-04$ &  $\phantom{-}6.29e-04$ \\ 
4 & $\phantom{-}2.51e-05$ &   $\phantom{-}2.51e-05$ \\ 
5 & $-4.77e-08$ & $-4.77e-08$ \\ 
6 & $-8.10e-10$ & $-8.10e-10$ \\ 
\hline
\end{tabular}
\end{center}
\end{table}

\end{example}

%
%

\subsection{Generalized Laguerre weight functions}
We consider the situation when the sequence of monic orthogonal polynomials
$\{p_m\}_{m=0}^\infty$ are generalized Laguerre polynomials~\cite{Gautschi04,MMlibro}, 
i.e., they satisfy \eqref{orthonormality} with respect to the domain $\cD=\mathbb{R}^+$ 
and the measure $d\mu(x)=x^\alpha e^{-x}dx$ for some $\alpha>-1$. The recursion 
coefficients are given by
$$
\begin{aligned}
\alpha_k &= 2k+\alpha+1, \qquad k\geq 0, \\
\beta_0 & =\Gamma(1+\alpha), \qquad \beta_k=k(k+\alpha), \qquad k\geq 1.
\end{aligned}
$$
It is easy to see that
\[
\frac{\beta_{m+1}}{\beta_m+\beta_{m+1}} \to \frac{1}{2}\mbox{~~~and~~~}
\frac{\beta_{m}}{\beta_m+\beta_{m+1}} \to \frac{1}{2}\mbox{~~~as~~~} m\to\infty.
\]


\begin{example}\rm
Regard the integral 
\[
I_2=\int_0^\infty \frac{1}{(x-2)^2+4} w(x)\, dx, \qquad w(x) = \sqrt{x} e^{-x},
\]
whose exact solution is approximated by a Gauss rule with 1024 nodes.
Table \ref{table:I3} displays quadrature errors for this integral. The averaged rules 
can be seen to yield one or two more correct decimal digits than the corresponding Gauss
rule. In this example the averaged rule produces higher accuracy than the weighted
averaged rule with the same number of nodes.

\begin{table}[ht]
\begin{center}
\caption{Quadrature errors for the integral $I_2$.}
\label{table:I3}
\begin{tabular}{c|ccc|cc}
\hline
$m$ & $I_2-G_m$ & $I_2-\widetilde{G}_{m+1}$ & $I_2-G^*_{m+1}$ 
& $I_2-\widetilde{A}_{2m+1}$ & $I_2-\widehat{A}_{2m+1}$ \\
\hline
8 &  $\phantom{-}2.55e-04$ & $-2.83e-04$ & $-1.92e-04$ & $-1.38e-05$ & 
$\phantom{-}5.72e-05$ \\ 
16 & $-4.40e-06$ & $\phantom{-}2.73e-06$ &  $\phantom{-}9.11e-06$ & $-8.37e-07$ & 
$\phantom{-}1.95e-06$ \\ 
32 & $\phantom{-}2.59e-07$ & $-2.44e-07$ & $-3.01e-07$ & $\phantom{-}7.39e-09$ &
$-1.27e-08$ \\ 
64 &  $\phantom{-}2.54e-10$ & $-2.76e-10$ & $-1.87e-10$ & $-1.10e-11$ & 
$\phantom{-}3.72e-11$ \\ 
128 & $-1.53e-13$ & $\phantom{-}1.51e-13$ & $\phantom{-}1.60e-13$ & $-1.33e-15$ & 
$\phantom{-}2.08e-15$ \\ 
\hline
\end{tabular}
\end{center}
\end{table}
\end{example}

\subsection{The Hermite weight function}
We consider the measure 
$$
d \mu(x)=e^{-x^2} dx.
$$
The monic Hermite
orthogonal polynomials satisfy \eqref{recurrence} with the coefficients
$$
\begin{aligned}
\alpha_k &= 0, \qquad k\geq 0, \\
\beta_0 & =\sqrt{\pi}, \qquad \beta_k=\frac{k}{2}, \qquad k\geq 1.
\end{aligned}
$$
Then
$$
\frac{\beta_{m+1}}{\beta_m+\beta_{m+1}}=\frac{m+1}{2m+1} \to \frac{1}{2},
\qquad
\frac{\beta_{m}}{\beta_m+\beta_{m+1}}=\frac{m}{2m+1} \to \frac{1}{2}
$$
as $m\to\infty$.

\begin{example}\rm
Consider the integral
$$
\begin{aligned}
I_3  &= \int_{\mathbb{R}} \cosh(x) \, w(x)\dx, \\
\end{aligned}
$$ 
with $w(x)=e^{-x^2}$. The exact value is approximated by a Gauss rule with 512 
nodes. The quadrature errors reported in Table~\ref{table:I4} show that also for the
integral in the present example, the averaged rules yield higher accuracy than the 
underlying Gauss rules, and the weighted averaged formula is superior to the averaged one.
This is due to the smoothness of the integrand.

\begin{table}[ht]
\begin{center}
\caption{Quadrature errors for the integral $I_3$.}
\label{table:I4}
\begin{tabular}{c|ccc|cc}
\hline
$m$ & $I_3-G_m$ & $I_3-\widetilde{G}_{m+1}$ & $I_3-G^*_{m+1}$ 
& $I_3-\widetilde{A}_{2m+1}$ & $I_3-\widehat{A}_{2m+1}$ \\
\hline
2 &  $4.15e-02$ & $-4.01e-02$ & $-6.22e-02$ &  $7.41e-04$ & $\phantom{-}5.64e-05$ \\ 
4 &  $7.41e-05$ & $-7.32e-05$ & $-9.26e-05$ &  $4.37e-07$ & $\phantom{-}2.39e-08$ \\ 
6 &  $4.69e-08$ & $-4.66e-08$ & $-5.46e-08$ &  $1.35e-10$ & $\phantom{-}5.76e-12$ \\ 
8 &  $1.50e-11$ & $-1.50e-11$ & $-1.69e-11$ &  $2.40e-14$ & $-8.88e-16$ \\ 
\hline
\end{tabular}
\end{center}
\end{table}


\end{example}

\section{Averaged and weighted averaged Nystr\"om-type interpolants}\label{sec:nystrom}
This section describes several ways to apply the averaged and weighted averaged
quadrature rules to compute and evaluate an approximate solution of the
integral equation \eqref{Fredholm} in suitable weighted spaces $C_u(\cD)$.
The consideration of the equations in weighted spaces is 
crucial in order to include the cases when the kernel, right-hand side, and 
the solution may be unbounded at some boundary points of the domain $\cD$.

Let the quadrature formula $K_m$ employed in \eqref{I+Km} be the
$m$-point Gauss rule $G_m$~\eqref{Gn}. Then, after multiplying both sides by $u(x_i^{(G)})$, the system \eqref{sys} becomes 
\begin{equation}\label{Gsys}
\sum_{j=1}^{m} \left[\delta_{ij}+ \lambda_j^{(G)}
\frac{u(x_i^{(G)})}{u(x_j^{(G)})} k(x_j^{(G)},x_i^{(G)})\right]
{a}_j^{(G)}=(g^{(G)}u)(x_i^{(G)}),
\qquad i=1,2,\dots,m,
\end{equation}
with $a^{(G)}_j=(f_m^{(G)}u)(x_j^{(G)})$, $1\leq j\leq m$, and we determine an 
approximate solution of \eqref{Fredholm} by using the weighted Nystr\"om interpolant
\begin{equation}\label{fm}
(f_{m}^{(G)}u)(y)=(g^{(G)}u)(y)-u(y)\sum_{j=1}^{m} \frac{\lambda_j^{(G)}}{u(x_j^{(G)})} k(x_j^{(G)},y)
a^{(G)}_j,\qquad y\in\cD.
\end{equation}
 We note for
future reference that the linear system of equations \eqref{Gsys} can be
expressed as
\begin{equation}\label{system2}
(I_m+D_m^{(G)} \Phi^{(G)}({D_m^{(G)}})^{-1})\bm{a}^{(G)}=\bm{g}^{(G)},
\end{equation}
where $I_m$ is the identity matrix of order $m$,
\[
D_m^{(G)}=\diag(u(x_1^{(G)}),\ldots,u(x_m^{(G)})),
\]
the unknown vector is
$\bm{a}^{(G)}=[a_1^{(G)},a_2^{(G)},\ldots,a^{(G)}_m]^T \in \RR^m$, 
and the entries of the matrix
$\Phi^{(G)}=[\phi^{(G)}_{ij}]\in\RR^{m\times m}$ and right-hand side vector
$\bm{g}^{(G)}=[g_i^{(G)}]\in\RR^m$ are given by
$$
\begin{aligned}
\phi^{(G)}_{ij}&=\lambda_j^{(G)} k(x^{(G)}_j,x^{(G)}_i),\qquad 
&& i,j=1,2,\ldots,m,\\
g_i^{(G)}&=(gu)(x_i^{(G)}),\qquad && i=1,2,\ldots,m.
\end{aligned}
$$
We tacitly assume that $m$ is large enough so that the system \eqref{system2}
has a unique solution; see~\cite[Theorem~4.1.2]{A}. The computation of this
solution by LU factorization of the matrix
$I_m+D_m^{(G)}\Phi^{(G)}({D_m^{(G)}})^{-1}$ requires about
$\frac{2}{3}m^3$ arithmetic floating point operations (flops); see,
e.g.,~\cite[Lecture 20]{TB}. 

A Nystr\"om method based on the anti-Gauss quadrature rule \eqref{Gm+1} was
recently described in~\cite{DFR2020}, where the interpolant corresponding to
the $(2m+1)$-point averaged quadrature rule \eqref{averaged1} was studied.

In the following, we discretize the integral equation \eqref{Fredholm} by the
$(2m+1)$-point weighted averaged quadrature rule that is associated
with the $m$-point Gauss rule, and define a corresponding weighted Nystr\"om
interpolant. This interpolant gives higher accuracy than the Nystr\"om
interpolant defined by the $m$-point Gauss rule. The difference between the
approximate solutions of \eqref{Fredholm} furnished by the Nystr\"om
interpolants associated with the $m$-point Gauss rule and the corresponding
$(2m+1)$-point weighted averaged quadrature formula is used to
estimate the error in the approximate solution obtained by the Gauss--Nystr\"om
interpolant. This approach of estimating the error is analogous to the
technique used in Section \ref{sec3}.

\subsection{A weighted averaged Nystr\"om interpolant}

We consider the Nystr\"om interpolant \eqref{eqfm} that is determined by the 
$(2m+1)$-point weighted averaged quadrature rule $\widehat{A}_{2m+1}$
\eqref{G2m+1} associated with the $m$-point Gauss rule \eqref{Gn} used in
\eqref{Gsys}. The determination of this interpolant requires the solution of
the equation
\begin{equation}\label{eqAhat}
(I+\widehat{K}_{2m+1})\widehat{f}_{2m+1}^{[1]}=g,
\end{equation}
with 
\begin{equation}\label{kappahat}
(\widehat{K}_{2m+1} f)(y)=\sum_{j=1}^{2m+1} \widehat{\lambda}_j
k(\widehat{x}_j,y) f(\widehat{x}_j),
\end{equation}
that is, the solution of the linear system of equations with a matrix of order
$2m+1$,
\begin{equation}\label{system1}
\sum_{j=1}^{2m+1} \left[\delta_{ij}+\widehat{\lambda}_j \frac{u(\widehat{x}_i)}{u(\widehat{x}_j)}
k(\widehat{x}_j,\widehat{x}_i)\right]\widehat{a}_j=(gu)(\widehat{x}_i), \qquad i=1,2,\dots,2m+1,
\end{equation}
with $\widehat{a}_j=(\widehat{f}_{2m+1}^{[1]}u)(\widehat{x}_j)$.
We assume as usual that $m$ is large enough so that this system has a unique solution 
$\widehat{\bm{a}}=[\widehat{a}_1,\widehat{a}_2,\ldots,\widehat{a}_{2m+1}]^T$, which 
determines the weighted averaged Nystr\"om interpolant
\begin{equation}\label{eqfhat}
(\widehat{f}_{2m+1}^{[1]}u)(y)=(gu)(y)-u(y)\sum_{j=1}^{2m+1} \frac{\widehat{\lambda}_j}{u(\widehat{x}_j)} k(\widehat{x}_j,y)
\widehat{a}_j, \qquad y \in \cD.
\end{equation}
We will use the difference $(\widehat{f}_{2m+1}^{[1]}(y)-f_m^{(G)}(y))u(y)$ 
as an estimate of the error in $(f_m^{(G)}u)(y)$.

Introduce the matrices
$D_{2m+1}=\diag(u(\widehat{x}_1),\ldots,u(\widehat{x}_{2m+1}))$ and
$\Phi=[\phi_{ij}]\in\RR^{(2m+1)\times(2m+1)}$
with entries 
\[
\phi_{ij}=[\widehat{\lambda}_j k(\widehat{x}_j,\widehat{x}_i)],\qquad i,j=1,2,\ldots,2m+1,
\]
and the vector $\bm{g}=[(gu)(\widehat{x}_i)]\in\RR^{2m+1}$. Then the linear
system of equations 
\eqref{system1} can be written as
\begin{equation}\label{phisys1}
(I_{2m+1}+D_{2m+1}\Phi D_{2m+1}^{-1})\widehat{\bm{a}}=\bm{g}.
\end{equation}
The solution of this linear system by LU factorization requires about
$\frac{16}{3}m^3$ flops. Thus, the total computational effort required to solve
both the systems \eqref{system2} and \eqref{phisys1} is about $\frac{18}{3}m^3$
flops.

\begin{theorem}\label{teo:1}
Assume that $\mathcal{N}(I+K)=\{0\}$ in $C_u(\cD)$ and let $f$ be the unique solution of 
equation \eqref{Fredholm} for each given right-hand side $g\in C_u(\cD)$. If 
$$\int_{\cD} \frac{d\mu(x)}{u^2(x) }< \infty$$
and the kernel function $k$ satisfies
$$\sup_{y \in \cD} u(y) \|k(\cdot,y)\|_{W^r_u(\cD)}< \infty, \qquad \sup_{x \in \cD} u(x) \|k(x,\cdot)\|_{W^r_u(\cD)}< \infty,$$
then equation \eqref{eqAhat} has a unique solution 
$\widehat{f}_{2m+1}^{[1]}\in C_u(\cD)$ for $m$ sufficiently large.

If, in addition, the right-hand side satisfies $g \in W_u^r(\cD)$, then 
$\| (f - \widehat{f}_{2m+1}^{[1]})u \|_\infty$ tends to zero as the error of best 
polynomial approximation in $W_u^r(\cD)$.

Finally, the $\infty$-norm condition number of the coefficient matrix 
in \eqref{phisys1}
is bounded independently of $m$, for $m$ sufficiently large.
\end{theorem}

\begin{proof}
To prove the first part of the theorem we have to show that
\begin{itemize}
\item[(A)] $\displaystyle \lim_{m \to \infty} \|(K-\widehat{K}_{2m+1})f u \|_\infty=0$ for
each $f \in C_u(\cD)$,
\item[(B)] $\displaystyle \lim_{M \to \infty} \displaystyle \left(\sup_{m} 
\left[\sup_{\|fu\| \leq 1} E_{M}(\widehat{K}_{2m+1}f)_u\right] \right)=0$.  
\end{itemize}
Let us first prove (A). We observe that the function $k(\cdot,y)f$ is in $C_{u^2}(\cD)$. By 
applying Corollary~\ref{stabAhat} to this function with $u^2$ in place of $u$, we deduce
\begin{equation*}
\left|(Kf)(y)-(\widehat{K}_{2m+1}f)(y)\right| \leq \mathcal{C} E_{2m-1}(k(\cdot,y) f)_{u^2}.
\end{equation*}
Therefore, multiplying both sides times $u(y)$ and
using~\cite[Equation 4.1]{FermoRusso},
\begin{equation*}
E_m(fg)_{u^2} \leq 2 \left[ \|fu\|_\infty E_m(g)_u 
+ \|gu\|_\infty E_m(f)_u \right],
\end{equation*}
we obtain
\begin{multline}\label{convergence}
\left|[(Kf)(y)-(\widehat{K}_{2m+1}f)(y)] u(y)\right| \leq \mathcal{C} \left[ \sup_{y \in \cD} u(y) \|
k(\cdot,y)u\|_\infty E_{m-1}(f)_u \right. \\ 
\left.+\|fu\|_\infty \sup_{y \in \cD} u(y) E_{m-1}(k(\cdot,y))_u \right],
\end{multline}
which tends to zero as $m$ increases by the assumptions on $k(\cdot,y)$ and $f$.

To prove (B), we show that $\widehat{K}_{2m+1}$ maps $C_u(\cD)$ into $W_u^r(\cD)$.
Indeed, from \eqref{kappahat},
\begin{multline*}
\left|(\widehat{K}_{2m+1}f)^{(r)}(y)\varphi^r(y) u(y) \right| \\
\begin{aligned}
&\leq \sum_{j=1}^{2m+1}
\frac{\widehat{\lambda}_j}{u^2(\widehat{x}_j)} u(\widehat{x}_j)
\left|k(\widehat{x}_j,y)^{(r)} \varphi^r(y) u(y)\right| \cdot |f(\widehat{x}_j)
u(\widehat{x}_j)| \\ 
&\leq \|fu\|_\infty \sup_{x \in \cD} u(x)
\left\|k(x,\cdot)^{(r)}\varphi^ru\right\|_\infty \sum_{j=1}^{2m+1}
\frac{\widehat{\lambda}_j}{u^2(\widehat{x}_j)},
\end{aligned}
\end{multline*}
which is bounded by virtue of the assumptions on the kernel and on the weight function, 
and by taking Corollary \ref{stabAhat} into account. Hence, 
$\widehat{K}_{2m+1}f\in W_u^r(\cD)$ and we can deduce that  
$$
E_M(\widehat{K}_{2m+1}f) \leq \frac{\mathcal{C}}{M^{c r}} \|fu\|_\infty,
$$
where $c$ and $\mathcal{C}$ are positive constants independent of $m$, $M$, and $f$. 
Condition (B) now follows and, consequently, 
$\|(K-\widehat{K}_{2m+1})\widehat{K}_{2m+1}\|$ tends to zero as $m$ increases. This can be
seen by \eqref{convergence} with $\widehat{K}_{2m+1}f$ in place of $f$ and by applying our 
assumnptions on the kernel function.
Therefore by \cite[Chapter 4]{A}, the operator $I-\widehat{K}_{2m+1}$ is invertible in 
$C_u(\cD)$, for $m$ sufficiently large, and uniformly bounded. This completes the proof of
the first assertion.

The estimate of the error follows from \cite[Theorem 4.1.2]{A},
$$
\left\| (f - \widehat{f}_{2m+1}^{[1]})u \right\|_\infty \sim
\left\|[Kf-\widehat{K}_{2m+1}f] u\right\|_\infty,
$$
and by applying \eqref{convergence}.

A proof of the well-conditioning of the linear system is given in~\cite[p.\,113]{A}. We 
just have to replace the usual infinity norm by the weighted norm of the space $C_u(\cD)$.
\end{proof}

\subsection{An approximate Nystr\"om interpolant based on the splitting
\eqref{G2m+1_bis}} \label{approxint}
The splitting \eqref{G2m+1_bis} suggests the use of an approximate Nystr\"om interpolant
that is cheaper to compute than solving \eqref{phisys1}. The evaluation of the
approximate interpolant proceeds as follows:
\begin{enumerate}
\item[I.] Determine and solve the
linear system of equations \eqref{system2} associated with the $m$-point Gauss
rule. This yields the Nystr\"om interpolant $f_m^{(G)}u$ defined by \eqref{fm}.

\item[II.] Compute the $(m+1)$-point Nystr\"om interpolant that is associated
with the quadrature rule $G^*_{m+1}$ in \eqref{Gstar}, with nodes and weights 
$x^*_j$ and $\lambda^*_j$, respectively. Thus, we consider the equation
\begin{equation}\label{eqstar}
(I+K^*_{m+1})f_{m+1}^*=g,
\end{equation}
where 
$$(K^*_{m+1}f)(y)=\sum_{j=1}^{m+1}\lambda^*_j  k(x^*_j,y) f(x^*_j).$$
This leads to the Nystr\"om interpolant
\begin{equation}\label{f*m+1}
(f^*_{m+1}u)(y)=(gu)(y)-u(y)\sum_{j=1}^{m+1} \frac{\lambda^*_j}{u(x^*_j)} k(x^*_j,y) a^*_j,\qquad y\in\cD,
\end{equation}
whose coefficients $a^*_j=(f_{m+1}^*u)(x^*_j)$, $j=1,2,\ldots,m+1$, are the
unknowns in the linear system of equations
\[
\sum_{j=1}^{m+1} \left[\delta_{ij}+ \lambda^*_j \frac{u(x^*_i)}{u(x^*_j)}
k(x^*_j,x^*_i)\right] {a}^*_j=(gu)(x^*_i),
\qquad i=1,2,\dots,m+1.
\]

We can express the above linear system in the form
\begin{equation}\label{phisys*}
(I_{m+1}+D_{m+1}^*\Phi^*(D_{m+1}^*)^{-1})\bm{a}^*=\bm{g}^*,
\end{equation}
where $\bm{a}^*=[a_1^*,\ldots,a_{m+1}^*]^T$,
\[
D_{m+1}^*=\diag(u(x^*_1),\ldots,u(x^*_{m+1})),
\]
and the entries of the matrix $\Phi^*=[\phi^*_{ij}]\in\RR^{(m+1)\times(m+1)}$
and right-hand side vector $\bm{g}^*=[g_i^*]\in\RR^{m+1}$ are given by
\[
\begin{aligned}
\phi^*_{ij}&=\lambda_j^* k(x^*_j,x^*_i),\qquad &&i,j=1,2,\ldots,m+1, \\
g_i^*&=(gu)(x_i^*)\quad &&i=1,2,\ldots,m+1.
\end{aligned}
\]

\item [III.] Approximate the solution of the original equation \eqref{Fredholm} by a 
convex combination of the weighted Nystr\"om interpolants computed in steps I and II,
as suggested by the representation \eqref{G2m+1_bis} of the quadrature rule 
\eqref{G2m+1}:
\begin{equation}\label{interp2}
(\widehat{f}^{[2]}_{2m+1}u)(y)=\frac{\beta_{m+1}}{\beta_m+\beta_{m+1}} (f_m^{(G)}u)(y)
+\frac{\beta_{m}}{\beta_m+\beta_{m+1}} (f^*_{m+1}u)(y).
\end{equation}
\end{enumerate}

The determination of this interpolant requires the solution of two linear
systems of equations with matrices of orders $m$ and $m+1$, respectively.
Their solution demands about $\frac{4}{3}m^3$ flops. In particular, this
includes the computational effort required to evaluate the Gauss--Nystr\"om
interpolant \eqref{fm}. Hence, the calculation of the interpolant
\eqref{interp2} is cheaper than the computation of the interpolant
\eqref{eqfhat}. We will compare the accuracy of the approximations
$\widehat{f}^{[1]}_{2m+1}$ and  $\widehat{f}^{[2]}_{2m+1}$, defined by \eqref{eqfhat} and
\eqref{interp2}, respectively, of the solution $f$ of \eqref{Fredholm} in 
Section~\ref{examples}.

We observed in Section~\ref{sec:formulae} that for Chebychev measures, the coefficients 
$\frac{\beta_{m+1}}{\beta_m+\beta_{m+1}}$ and $\frac{\beta_m}{\beta_m+\beta_{m+1}}$ in
\eqref{G2m+1_bis} are $\frac{1}{2}$. When $m$ tends to $\infty$, these coefficients tend 
to $\frac{1}{2}$ as $m$ increases also for other measures. When the coefficients equal
$\frac{1}{2}$, the Nystr\"om interpolant \eqref{interp2} coincides with the averaged 
interpolant determined by the Gauss and anti-Gauss rules. The latter interpolant has been
investigated in~\cite{DFR2020}.

We conclude this subsection by showing convergence and stability of the Nystr\"om method 
based on the quadrature rule $G_{m+1}^*$. This yields \eqref{f*m+1}, and convergence of 
the interpolant \eqref{interp2}.

\begin{theorem}\label{teo:2}
Assume that $\mathcal{N}(I+K)=\{0\}$ in $C_u(\cD)$ and let $f$ be the unique solution of 
equation \eqref{Fredholm} for the right-hand side function $g\in C_u(\cD)$. If 
$$\int_{\cD} \frac{d\mu(x)}{u^2(x) }< \infty,$$
and the kernel function $k$ is such that
$$\sup_{y \in \cD} u(y) \|k(\cdot,y)\|_{W^r_u(\cD)}< \infty, \qquad 
\sup_{x \in \cD} u(x) \|k(x,\cdot)\|_{W^r_u(\cD)}< \infty,$$
then, for $m$ large enough, equation \eqref{eqstar} has a unique solution 
${f}_{m+1}^{*}\in C_u(\cD)$. Moreover, if the right-hand side $g$ is in $ W_u^r(\cD)$, 
then $\left\| (f - {f}_{m+1}^*)u \right\|_\infty$ tends to zero as $m$ increases as the
error of best polynomial approximation in $W_u^r(\cD)$.

Finally, the condition number of the matrix $(I_{m+1}+\Phi^*)$ in the $\infty$-norm is 
bounded independently of $m$, for $m$ sufficiently large.
\end{theorem}

\begin{proof}
The assertions can be proved similarly as Theorem~\ref{teo:1}, by applying Corollary 
\ref{corollary} in place of Corollary \ref{stabAhat}.
\end{proof}

\smallskip
\begin{proposition}
Assume that $\mathcal{N}(I+K)=\{0\}$ in $C_u(\cD)$ and let $f$ be the unique solution of 
equation \eqref{Fredholm} for the right-hand side function $g\in C_u(\cD)$. Then, under 
the assumption of Theorem \ref{stability},
$$\lim_{m \to \infty} \|(f-\widehat{f}^{[2]}_{2m+1})u\|_\infty=0,$$
where $\widehat{f}^{[2]}_{2m+1}$ is given by \eqref{interp2}.
\end{proposition}

\begin{proof}
By \eqref{interp2} we have
$$
\|(f-\widehat{f}^{[2]}_{2m+1})u\|_\infty \leq \frac{\beta_{m+1}}{\beta_m+\beta_{m+1}} 
\|(f-f_m^{(G)})u \|_\infty +\frac{\beta_{m}}{\beta_m+\beta_{m+1}} 
\|(f-f^*_{m+1})u\|_\infty,
$$
from which the assertion follows by considering that the coefficients tend to 
$\frac{1}{2}$ as $m \to \infty$ and by taking Theorems~\ref{teo:1} and \ref{teo:2} into 
account.
\end{proof}

\section{Iterative methods for the evaluation of the Nystr\"om interpolant 
\eqref{eqfhat}}\label{itmeth}
This section describes several iterative methods for computing
approximations of the Nystr\"om interpolant \eqref{eqfhat} that are more
accurate than the approximation described in the previous subsection.
This paper discusses some simple algorithms directly stemming from the
quadrature rules used to compute \eqref{eqfhat}. We prove their convergence
under the assumption that the weight of the space $C_u(\cD)$ is $u(x)=1$.
However, in Section \ref{examples}, we show by a numerical experiment that a
suitable choice of the weight $u$ may improve the rate of convergence. Other
iterative methods also could be employed, such as Krylov methods. Here, we
focus on methods that exploit the structure of the problem.

Using the notation introduced above, we set
\begin{eqnarray*}
\theta_m^{(1)}&=&\frac{\beta_{m+1}}{\beta_m+\beta_{m+1}}, \\
\theta_m^{(2)}&=&\frac{\beta_m}{\beta_m+\beta_{m+1}},\\
\Phi_{11}&=&\theta_m^{(1)} \Phi^{(G)},\\
\Phi_{22}&=&\theta_m^{(2)} \Phi^*,
\end{eqnarray*}
see \eqref{system2} and \eqref{phisys*}, and
$$
\begin{aligned}
(\Phi_{12})_{ij}&=\theta_m^{(2)}
\lambda_j^* k(x_j^*,x_i^{(G)}), & \qquad &i=1,2,\ldots,m,& j&=1,2,\ldots,m+1, \\
(\Phi_{21})_{ij}&=\theta_m^{(1)}
\lambda_j^{(G)} k(x_j^{(G)},x_i^*), & \qquad &i=1,2,\ldots,m+1,& j&=1,2,\ldots,m.
\end{aligned}
$$
Express the system \eqref{phisys1} as
\begin{equation}\label{split}
\begin{bmatrix}
I_m+\Phi_{11} & \Phi_{12} \\
\Phi_{21} & I_{m+1}+\Phi_{22} \end{bmatrix} 
\begin{bmatrix}\bm{b} \\ \bm{c} \end{bmatrix} =
\begin{bmatrix}\bm{g}^{(G)} \\ \bm{g}^* \end{bmatrix},
\end{equation}
where $\bm{b}=\bm{a}^{(G)}\in\RR^m$, $\bm{c}=\bm{a}^{*}\in\RR^{m+1}$.
The representation \eqref{split} suggests a few iterative solution methods.
The first one we consider is a modification of the method considered in
Subsection \ref{approxint}, which uses the computed LU factorizations in an
iterative fashion.
It is defined by
\begin{equation}\label{iter1}
\begin{aligned}
(I_m+\Phi_{11}) \bm{b}^{(k+1)} &= \bm{g} - \Phi_{12} \bm{c}^{(k)}, \\
(I_{m+1}+\Phi_{22}) \bm{c}^{(k+1)} &= \bm{g}^* - \Phi_{21} \bm{b}^{(k+1)},
\end{aligned}
\qquad k=0,1,2,\ldots.
\end{equation}

Since the method is stationary, the LU factorizations of the coefficient matrices can be
computed initially, and then used in each iteration. Computing the vector $\bm{c}^{(0)}$ 
by the $(m+1)$-point $G_{m+1}^*$ quadrature formula yields a quite accurate initial 
approximation of the second part of the solution. In actual computations, convergence is
typically achieved within a fairly small number of iterations. The following results give 
sufficient conditions for the convergence of the method.

\medskip
\begin{theorem}\label{thm1}
Let the kernel of \eqref{Fredholm} satisfy 
\begin{equation}\label{kinfty}
\|k\|_\infty = \sup_{y\in\cD}\|k(\cdot,y)\|_\infty < \beta_0^{-1},
\end{equation}
where $\beta_0=\langle p_0,p_0 \rangle$. Then, for a sufficiently large $m$, the iteration
process \eqref{iter1} converges to the vectors
\begin{equation}\label{itersol}
\bm{b} = [\hat{f}_{2m+1}^{[1]}(x_1^{(G)}), \ldots, 
	\hat{f}_{2m+1}^{[1]}(x_m^{(G)}) ]^T, \quad
\bm{c} = [ \hat{f}_{2m+1}^{[1]}(x_1^*), \ldots, 
	\hat{f}_{2m+1}^{[1]}(x_{m+1}^*) ]^T,
\end{equation}
that is, to the unique solution of system \eqref{phisys1}.
\end{theorem}

\begin{proof}
Let $\bm{b}\in\RR^m$ and $\bm{c}\in\RR^{m+1}$ denote the solution of \eqref{iter1}. 
Introduce the error vectors $\bm{e}_b^{(k+1)}=\bm{b}^{(k+1)}-\bm{b}$ and 
$\bm{e}_c^{(k)}=\bm{c}^{(k+1)}-\bm{c}$ for $k=0,1,2,\ldots~$. The assumption 
\eqref{kinfty} implies that there exists a constant $\varepsilon>0$ such that the matrix
$\Phi_{11}$ satisfies
$$
\|\Phi_{11}\|_\infty = \theta_m^{(1)} \max_{i=1,2,\ldots,m} 
\sum_{j=1}^m \lambda_j^{(G)}  \vert k(x_j^{(G)},x_i^{(G)}) \vert 
\leq \theta_m^{(1)} \beta_0 \|k\|_\infty < \theta_m^{(1)}-\varepsilon < 
\frac{1}{2},
$$
since the sum of the quadrature weights is $\beta_0$.
Similarly, we obtain
\begin{equation}\label{unmezzo}
\|\Phi_{ij}\|_\infty< \frac{1}{2}, \qquad i,j=1,2.
\end{equation}
Therefore, the matrices $I_m+\Phi_{11}$ and $I_{m+1}+\Phi_{22}$ are invertible
and
$$
\|(I_m+\Phi_{\ell\ell})^{-1}\|_\infty \leq
\frac{1}{1-\|\Phi_{\ell\ell}\|_\infty}, \qquad \ell=1,2.
$$
Combining \eqref{iter1} and \eqref{split}, we obtain   
\[
\bm{e}_b^{(k+1)}=-(I_m+\Phi_{11})^{-1}\Phi_{12}\bm{e}_c^{(k)},\qquad
\bm{e}_c^{(k+1)}=-(I_m+\Phi_{22})^{-1}\Phi_{21}\bm{e}_b^{(k+1)},
\]
which yields
\[
\bm{e}_c^{(k+1)}=(I_m+\Phi_{22})^{-1}\Phi_{21}
(I+\Phi_{11})^{-1}\Phi_{12}\bm{e}_c^{(k)},\quad k=0,1,2,\ldots~.
\]
It follows from \eqref{unmezzo} that
$$
\|(I_m+\Phi_{11})^{-1}\Phi_{12}\|_\infty \leq 
\frac{\|\Phi_{12}\|_\infty}{1-\|\Phi_{11}\|_\infty} <1.
$$
Similarly, $\|(I_m+\Phi_{22})^{-1}\Phi_{21}\|_\infty<1$, and we obtain
\[
\|(I_m+\Phi_{22})^{-1}\Phi_{21}(I+\Phi_{11})^{-1}\Phi_{12}\|\infty <1.
\]
This shows that $\lim_{k\rightarrow\infty}\bm{e}_c^{(k)}=\bm{0}$. The fact that
$\lim_{k\rightarrow\infty}\bm{e}_b^{(k)}=\bm{0}$ can be shown similarly.
\end{proof}

\medskip
\begin{theorem}\label{thm1bis}
Let the kernel of \eqref{Fredholm} satisfy 
$$
\sup_{y\in\cD}\|k(\cdot,y)\|_1 \leq M < 1,
$$
where $M$ is a positive constant.
Then, for a sufficiently large $m$, the iteration process \eqref{iter1}
converges to the unique solution \eqref{itersol} of system
\eqref{phisys1}.
\end{theorem}

\begin{proof}
Proceeding as in the proof of Theorem~\ref{thm1}, we can write
$$
\|\Phi_{11}\|_\infty = \theta_m^{(1)} \max_{i=1,\ldots,m} 
\sum_{j=1}^m \lambda_j^{(G)}  \vert k(x_j^{(G)},x_i^{(G)}) \vert 
= \theta_m^{(1)} G_m( \vert k(\cdot,y) \vert ).
$$
Since $G_m$ is convergent, there are a sequence of constants $\rho_m>0$, $m=1,2,\ldots~$,
which converge to $1$, such that
$$
G_m( \vert k(\cdot,y) \vert ) \leq \rho_m \|k(\cdot,y)\|_1 \leq \rho_m M < 1,
$$
for $m$ sufficiently large. Recalling that $\theta_m^{(1)}$ converges to $\frac{1}{2}$ as 
$m$ increases, we have
$$
\|\Phi_{11}\|_\infty \leq \theta_m^{(1)} \rho_m M < \frac{1}{2}.
$$
The rest of the proof is similar to the proof of Theorem \ref{thm1}. 
\end{proof}

\medskip
The iterations \eqref{iter1} are terminated when two consecutive iterates are 
sufficiently close, that is when
$$
\|\bm{b}^{(k+1)}-\bm{b}^{(k)}\|_2<\tau \quad \text{and} \quad
\|\bm{c}^{(k+1)}-\bm{c}^{(k)}\|_2<\tau, 
$$
for a chosen tolerance $\tau$, or when a prescribed maximum number of iterations has been
carried out. We denote the computed solution by $\bm{b}^{(\rm{iter})}=[b_i^{(\rm{iter})}]$
and $\bm{c}^{(\rm{iter})}=[c_i^{(\rm{iter})}]$. This yields the interpolant
\begin{equation*}
\widehat{f}_{2m+1}^{[3]}(y) = g(y) 
-\theta_m^{(1)} \sum_{j=1}^{m}\lambda_j^{(G)} k(x_j^{(G)},y) 
b_j^{(\rm iter)} 
-\theta_m^{(2)} \sum_{j=1}^{m+1}\lambda_j^* k(x_j^*,y) 
c_j^{(\rm iter)},
\end{equation*}
for any $y\in\cD$.

We next consider the alternative iterative method
\begin{equation}\label{iter2}
\begin{aligned}
(I_m+\Phi_{11}) \bm{b}^{(k+1)} &= \bm{g} - \Phi_{12} \bm{c}^{(k)}, \\
\bm{c}^{(k+1)} &= \bm{g}^* - \Phi_{21}\bm{b}^{(k+1)} - \Phi_{22}\bm{c}^{(k)},
\end{aligned}
\qquad k=0,1,2,\ldots,
\end{equation}
which only requires the LU factorization of the matrix $I_m+\Phi_{11}$. Similarly,
for the method \eqref{iter1}, the initial solution $\bm{c}^{(0)}$ is computed by the
quadrature rule $G_{m+1}^*$. We denote the Nystr\"om interpolant corresponding to the 
above iteration by $\widehat{f}^{[4]}$.

\begin{theorem}\label{thm2}
Under the assumptions of Theorem~\ref{thm1}, the iteration process
\eqref{iter2} converges to the vectors \eqref{itersol}.
\end{theorem}

\begin{proof}
The error vectors for the method are
\[
\begin{aligned}
\bm{e}_b^{(k+1)} &= (I_m+\Phi_{11})^{-1}\Phi_{12} \left(
\Phi_{21} \bm{e}_b^{(k)} + \Phi_{22} \bm{e}_c^{(k-1)} \right), \\
\bm{e}_c^{(k+1)} &= \left( \Phi_{21} (I_m+\Phi_{11})^{-1}\Phi_{12} 
-\Phi_{22} \right) \bm{e}_c^{(k)}, 
\end{aligned}
\qquad k=1,2,\ldots~.
\]
The convergence can be easily proved by the same arguments as in the proof of
Theorem \ref{thm1}. 
\end{proof}

The last iterative method we consider is a Richardson-type method. It is
defined by
\begin{equation}\label{iter3}
\begin{aligned}
\bm{b}^{(k+1)} &= \bm{g} - \Phi_{11} \bm{b}^{(k)} -\Phi_{12} \bm{c}^{(k)}, \\
\bm{c}^{(k+1)} &= \bm{g}^* - \Phi_{21} \bm{b}^{(k+1)} 
-\Phi_{22} \bm{c}^{(k)},
\end{aligned}
\qquad k=0,1,2,\ldots~.
\end{equation}
Convergence can be established similarly as above. The LU factorization of a matrix is 
not required in this case, but the convergence is the slowest among the iterative methods
considered. We refer to the Nystr\"om interpolant computed in this manner as
$\widehat{f}^{[5]}$.

\begin{remark}\rm
We note that the iterative methods \eqref{iter1}, \eqref{iter2}, and \eqref{iter3} may be 
implemented by replacing $\bm{b}^{(k+1)}$ by $\bm{b}^{(k)}$ in the right-hand side of 
the second equation of each method. This has the advantage that the two formulas can be 
evaluated simultaneously on a parallel computer, but it slightly decreases the rate of
convergence.
\end{remark}

\section{Numerical examples}\label{examples}
To illustrate the performance of the averaged Nystr\"om interpolants discussed in the 
previous sections, we apply them to the solution of three integral equations with 
different degrees of regularity. The computations were carried out in Matlab R2021b on a 
Debian GNU/Linux computer.

In the tables that report the results, the symbols 
$$
R_m^{(G)},\ \tilde{R}_{m+1},\ R^*_{m+1},\ R_{2m+1}^{(A)},\ R_{2m+1}^{[i]},
\ i=1,2,\ldots,5,
$$
denote the difference in the uniform norm between the exact solution $f$ and the 
Nystr\"om interpolants 
$$
f_m^{(G)},\ \tilde{f}_{m+1},\ f^*_{m+1},\ f_{2m+1}^{(A)},\ f_{2m+1}^{[i]},
\ i=1,2,\ldots,5,
$$
respectively. We approximate the uniform norm by evaluating the error at $10^3$ 
equispaced points in the interval $(-1,1)$.

\begin{figure}[htb]
\begin{minipage}{.49\textwidth}
\begin{center} 
\includegraphics[width=\linewidth]{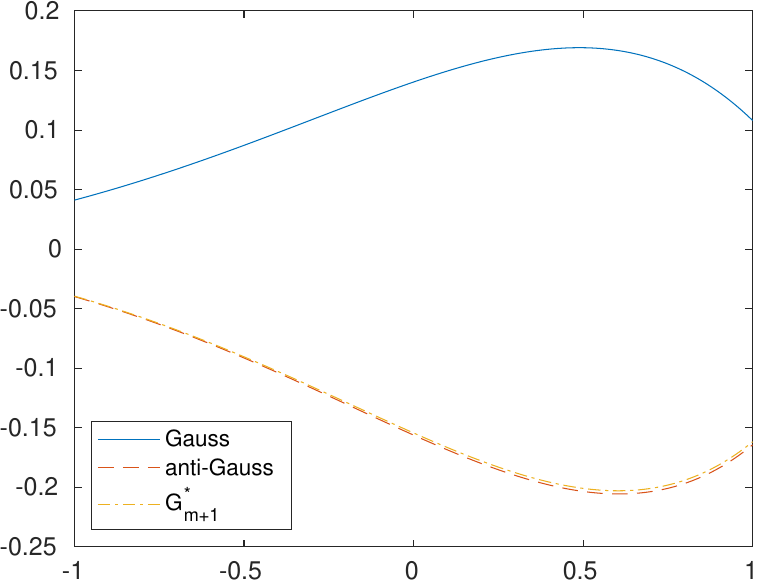}
\end{center}
\end{minipage}
\begin{minipage}{.49\textwidth}
\begin{center} 
\includegraphics[width=\linewidth]{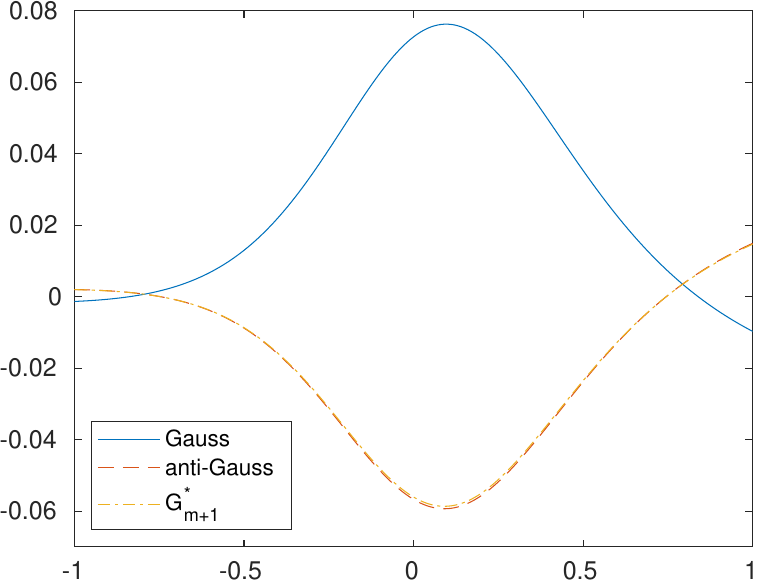}
\end{center}
\end{minipage}
\caption{Errors $f-f_m^{(G)}$ (Gauss rule), $f-\tilde{f}_{m+1}$ (anti-Gauss
rule), and $f-f_{m+1}^*$ ($G^*_{m+1}$ rule), for Example \ref{example1} (left)
and Example \ref{example2} (right) with $m=2$.}
\label{fig:ibmch}
\end{figure}

\begin{example}\label{example1}\rm
We first consider the equation 
\begin{equation}\label{es1}
f(y)+\frac{1}{2}\int_{-1}^1 x e^y\sin (x+y) f(x) \dx =g(y),
\end{equation}
where $g(y)=\frac{1}{32}(8\cos{2}-4\cos{4}-4\sin{2}+\sin{4})e^y\cos
y+\cos(3y)$. The exact solution is $f(y)=\cos{3y}$. 
We set $\alpha=\beta=0$ in the Jacobi weight function \eqref{jw} and
$\gamma=\delta=0$ in the weight $u(x)$ \eqref{weightu} of the function space
$C_u([-1,1])$.

In the graph on the left of Figure~\ref{fig:ibmch}, we plot the difference
between the exact solution $f$ and the approximation produced by the Gauss
rule, the anti-Gauss formula, and the quadrature rule $G^*_{m+1}$ for $m=2$.
We can observe that the errors of the second and third Nystr\"om interpolants are of 
opposite sign as the error in the Nystr\"om interpolant determined by the Gauss rule 
for all $x\in[-1,1]$.

\begin{table}[ht]
\begin{center}
\caption{Approximation errors for Example \ref{example1}.}
\label{es1:errors}
\begin{tabular}{c|cccccc}
\hline
$m$ & $R_m^{(G)}$ & $\tilde{R}_{m+1}$ & $R^*_{m+1}$ & $R_{2m+1}^{(A)}$ & $R_{2m+1}^{[1]}$ & $R_{2m+1}^{[2]}$ \\ 
\hline
   2 & 1.11e-01 & 1.26e-01 & 1.25e-01 & 1.10e-02 & 2.22e-03 & 1.20e-02 \\
   4 & 6.03e-03 & 6.03e-03 & 6.00e-03 & 2.42e-06 & 2.89e-07 & 3.57e-07 \\
   6 & 1.49e-05 & 1.49e-05 & 1.49e-05 & 6.88e-10 & 4.71e-11 & 4.69e-11 \\
   8 & 8.01e-09 & 8.01e-09 & 8.00e-09 & 9.53e-14 & 3.16e-15 & 3.77e-15 \\
  10 & 1.46e-12 & 1.46e-12 & 1.46e-12 & 3.33e-16 & 8.88e-16 & 2.22e-16 \\
\hline
\end{tabular}
\end{center}
\end{table}

Table~\ref{es1:errors} shows the behavior of the Nystr\"om interpolants determined by 
three quadrature rules (Gauss, anti-Gauss, and $G^*_{m+1}$), and compares them to the
interpolants determined by averaged and weighted averaged Gauss rules ($f_{2m+1}^{(A)}$
and $f_{2m+1}^{[1]}$), and the approximation $f_{2m+1}^{[2]}$ of the latter. The number 
of quadrature nodes, $m$, ranges from 2 to 10.

It can be seen that while the simple rules are equivalent, the averaged rules lead to 
improved accuracy, in some cases the improvement is 6 significant decimal digits. 
The weighted averaged rule 
is always more accurate than the averaged rule, except when machine precision is reached, 
where the larger linear system to be solved increases error propagation. The approximated 
interpolant $f_{2m+1}^{[2]}$ \eqref{interp2} is less accurate than $f_{2m+1}^{[1]}$ for 
small $m$, while it is equivalent for larger numbers of collocation points, and is the
most accurate for $m=10$. It can be seen that in the last case, the solution of two linear 
systems of orders $m$ and $m+1$, instead of the solution of one system of order $2m+1$, is 
beneficial both with respect to the complexity and stability of the numerical method.

\begin{table}[ht]
\begin{center}
\caption{Approximation errors for Example \ref{example1}.}
\label{es1:errorsiter}
\begin{tabular}{c|llll}
\hline
$m$ & $R_{2m+1}^{[1]}$ & $R_{2m+1}^{[3]}$ ($N_\text{iter}$) & $R_{2m+1}^{[4]}$ ($N_\text{iter}$) & $R_{2m+1}^{[5]}$ ($N_\text{iter}$) \\ 
\hline
   2 & 2.22e-03 & 2.22e-03 (13) & 2.22e-03 (21) & 2.22e-03 (25) \\ 
   4 & 2.89e-07 & 2.89e-07 (12) & 2.89e-07 (21) & 2.89e-07 (23) \\ 
   6 & 4.71e-11 & 4.71e-11 (10) & 4.71e-11 (17) & 4.71e-11 (20) \\ 
   8 & 3.16e-15 & 3.72e-15  (8) & 3.72e-15 (13) & 3.72e-15 (14) \\ 
  10 & 8.88e-16 & 1.11e-16  (5) & 1.11e-16  (8) & 1.11e-16  (8) \\ 
\hline
\end{tabular}
\end{center}
\end{table}

Table~\ref{es1:errorsiter} compares the weighted averaged interpolant
$f_{2m+1}^{[1]}$ computed by a direct method, to the approximations
$f_{2m+1}^{[i]}$, $i=3,4,5$, obtained by the iterative methods \eqref{iter1},
\eqref{iter2}, and \eqref{iter3}.
We remark that the kernel of \eqref{es1} satisfies the convergence assumption of
Theorem~\ref{thm1bis}, but not that of Theorem~\ref{thm1}.
For each algorithm, we report the uniform norm error and the number of
iterations necessary to reach convergence.
We see that the accuracy achieved by the three iterative methods is equivalent to the
accuracy obtained by direct solution. The iterative methods are more accurate for $m=10$.

In this test, we used a rather small tolerance $\tau=10^{-15}$ to stop the iteration. We
recall that the first iterative method requires two LU factorizations of matrices of 
orders $m$ and $m+1$, the second one just demands one LU factorization of a matrix of
order $m$, while \eqref{iter3} does not require the computation of a factorization.
Since the number of iterations for the second method is less than twice the number 
required by the first method, its computational cost is less. The third method has the 
lowest complexity, as it only involves matrix-vector product evaluations. However, since 
the matrix size is very small, it is not possible to actually measure the computing time.
\end{example}

\begin{example}\label{example2}\rm
The second integral equation is 
\begin{equation}\label{ex2}
f(y)+\int_{-1}^1 \frac{e^{x+y}}{1+x^2+3y^2} f(x)\sqrt[4]{1-x^2}dx  =
 \vert y+1 \vert ^{\frac{3}{2}},
\end{equation}
and the parameters of the Jacobi weight function \eqref{jw} are 
$\alpha=\beta=\frac{1}{4}$, and the space is not weighted, i.e., $\gamma=\delta=0$. The 
exact solution is not available. We therefore consider the Nystr\"om interpolant computed 
by a Gauss rule with $m=512$ the exact solution. The graph on the left of 
Figure~\ref{fig:ibmch} shows that, also in this case, the errors in the Nystr\"om 
interpolant determined by the anti-Gauss and $G^*_{m+1}$ rules with $m=2$ are of opposite
sign as the error in the Nystr\"om interpolant determined by the Gauss rule $G_m$ at all 
point of the interval $(-1,1)$.

\begin{table}[ht]
\begin{center}
\caption{Approximation errors for Example \ref{example2} ($\tau=10^{-15}$,
$\gamma=\delta=0$).}
\label{es2:errorsiter}
\begin{tabular}{c|lllll}
\hline
$m$ & $R_{2m+1}^{[1]}$ & $R_{2m+1}^{[2]}$ & $R_{2m+1}^{[3]}$ ($N_\text{iter}$) & $R_{2m+1}^{[4]}$ ($N_\text{iter}$) & $R_{2m+1}^{[5]}$ ($N_\text{iter}$) \\ 
\hline
   2 & 1.32e-03 & 8.18e-03 & 1.32e-03 (20) & 1.32e-03 (44) & 1.32e-03 (92) \\ 
   4 & 8.82e-06 & 1.33e-04 & 8.82e-06 (19) & 8.82e-06 (36) & 8.82e-06 (83) \\ 
   8 & 4.76e-09 & 2.49e-08 & 4.76e-09 (16) & 4.76e-09 (30) & 4.76e-09 (69) \\ 
  16 & 9.99e-11 & 9.99e-11 & 9.99e-11 (13) & 9.99e-11 (23) & 9.99e-11 (53) \\ 
  32 & 2.44e-12 & 2.44e-12 & 2.44e-12 (12) & 2.44e-12 (20) & 2.44e-12 (43) \\ 
  64 & 5.62e-14 & 5.60e-14 & 5.60e-14 (10) & 5.62e-14 (18) & 5.60e-14 (35) \\ 
 128 & 1.33e-15 & 2.39e-15 & 1.33e-15 (10) & 1.33e-15 (14) & 1.55e-15 (100) \\ 
 256 & 1.55e-15 & 1.01e-15 & 7.77e-16 (7) & 7.77e-16 (10) & 8.88e-16 (20) \\ 
\hline
\end{tabular}
\end{center}
\end{table}

Table~\ref{es2:errorsiter} compares the different algorithms for computing the weighted 
averaged interpolants $f_{2m+1}^{[i]}$, $i=1,2,\ldots,5$. The stop tolerance for the 
iterative methods is $\tau=10^{-15}$, and $\gamma=\delta=0$ for the space $C_u([-1,1])$.
The kernel of \eqref{ex2} does not satisfy the assumptions of Theorems~\ref{thm1} 
and~\ref{thm1bis}, but nevertheless the iterative methods converge, except for in one case.
As the right-hand side of \eqref{ex2} is not smooth, a large value of $m$ is required to 
achieve high accuracy. This is illustrated by the second column. Due to propagation of 
round-off errors, the error in the interpolants $f_{2m+1}^{[1]}$ does not decrease as $m$
becomes  larger than 128. The approximation $f_{2m+1}^{[2]}$ is slightly less accurate 
for $m<16$, but it is about the same as in $f_{2m+1}^{[1]}$ when $m$ is large. We recall that
the evaluation of $f_{2m+1}^{[2]}$ is cheaper than the evaluation of $f_{2m+1}^{[1]}$.

The iterative methods prove to be more stable than the direct approaches, reaching machine 
precision for $m=256$; see the last three columns in Table~\ref{es2:errorsiter}. The 
number of iterations decreases as the size of the problem increases, and the second 
iterative method appears to be more efficient than the first one, as the complexity is 
reduced. The accuracy of the third method is comparable, but the number if iterations 
grows with $m$, and exceeds the maximum number of allowed iterations (100) when $m=128$.

To investigate the influence of the weight in the space $C_u([-1,1])$, we repeat in 
Table~\ref{es2:errorsiterbis} the computations for the larger values of $m$ and
$\gamma=\delta=1.24$ in \eqref{weightu}. By \eqref{abcond}, $\gamma$ and $\delta$ should be less than $5/4$; we choose their value to be slightly smaller than this bound.
We observe that the accuracy is unaltered for the three iterative approaches,
but the number of iterations decreases, showing that the weight $u(x)$ acts as
a preconditioner for the iterative algorithms.
At the same time, the accuracy of the first direct method improves when $m=256$.

\begin{table}[ht]
\begin{center}
\caption{Approximation errors for Example \ref{example2} ($\tau=10^{-15}$,
$\gamma=\delta=1.24$).}
\label{es2:errorsiterbis}
\begin{tabular}{c|lllll}
\hline
$m$ & $R_{2m+1}^{[1]}$ & $R_{2m+1}^{[2]}$ & $R_{2m+1}^{[3]}$ ($N_\text{iter}$) & $R_{2m+1}^{[4]}$ ($N_\text{iter}$) & $R_{2m+1}^{[5]}$ ($N_\text{iter}$) \\ 
\hline
  32 & 1.80e-12 & 1.80e-12 & 1.80e-12 (11) & 1.80e-12 (19) & 1.80e-12 (42) \\ 
  64 & 4.20e-14 & 4.21e-14 & 4.19e-14 (9) & 4.19e-14 (16) & 4.19e-14 (32) \\ 
 128 & 1.33e-15 & 2.25e-15 & 1.55e-15 (7) & 1.33e-15 (12) & 1.33e-15 (24) \\ 
 256 & 7.77e-16 & 1.11e-15 & 7.77e-16 (5) & 8.88e-16 (8) & 8.88e-16 (16) \\ 
\hline
\end{tabular}
\end{center}
\end{table}

\end{example}

\begin{example}\label{example3}\rm
We apply the Nystr\"om method to the integral equation
\[
f(y)+\int_{-1}^1 (y+3) \vert \cos (3+x) \vert ^{\frac{5}{2}} f(x) (1-x)^{-\frac{1}{4}}
(1+x)^{\frac{4}{5}} \dx 
= \ln(1+y^2),
\]
and set $\alpha=-\frac{1}{4}$ and $\beta=\frac{4}{5}$ in the Jacobi weight function 
\eqref{jw}. Moreover, $\gamma=\delta=0$. The kernel does not satisfy the convergence 
conditions. 

\begin{table}[ht]
\begin{center}
\caption{Approximation errors for Example \ref{example3} ($\tau=10^{-12}$).}
\label{es3:errorsiter}
\begin{tabular}{c|llll}
\hline
$m$ & $R_{2m+1}^{[1]}$ & $R_{2m+1}^{[3]}$ ($N_\text{iter}$) & $R_{2m+1}^{[4]}$ ($N_\text{iter}$) & $R_{2m+1}^{[5]}$ ($N_\text{iter}$) \\ 
\hline
   2 & 9.67e-05 & 9.67e-05 (43) & 9.67e-05 (100) & 4.06e+42 (100) \\ 
   4 & 4.97e-08 & 4.97e-08 (38) & 4.97e-08 (77) & 2.68e+41 (100) \\ 
   8 & 9.35e-12 & 9.59e-12 (25) & 8.61e-12 (51) & 8.31e+37 (100) \\ 
  16 & 1.11e-16 & 4.79e-14 (3) & 4.97e-13 (3) & 2.74e+31 (100) \\ 
  32 & 2.22e-16 & 1.39e-16 (3) & 2.22e-16 (3) & 2.22e-16 (2) \\ 
  64 & 1.67e-16 & 1.67e-16 (3) & 1.11e-16 (3) & 2.78e-16 (2) \\ 
\hline
\end{tabular}
\end{center}
\end{table}

Table~\ref{es3:errorsiter} reports the results obtained for the direct and iterative 
methods for computing the weighted averaged interpolant. The first two iterative methods 
exhibit very slow convergence for $m\leq 8$, while they are fast for larger values 
of $m$. The third method diverges for $m\leq 16$, but converges in just 2 iterations for
larger values of $m$. This illustrates that convergence may occur only when $m$ is large
enough. This is not a significant drawback, since iterative methods are mainly intended
for medium and large-scale problems.
\end{example}

\section{Conclusions}\label{sec5}
This paper discusses and compares Nystr\"om interpolants determined by Gauss, averaged 
Gauss, and weighted averaged Gauss quadrature rules with a focus on the latter. 
Stability and accuracy of the Gauss rules used is investigated, and convergence 
of the Nystr\"om interpolants, and of iterative methods for their computation, are
discussed. For many problems the interpolants 
based on averaged Gauss and weighted averaged Gauss rules are shown to perform well. A complete analysis of the iterative methods for a generic weight function $u$ will be the subject of future research.

\section*{Acknowledgements}
Luisa Fermo and Giuseppe Rodriguez are members of INdAM-GNCS group. 
Luisa Fermo and Giuseppe Rodriguez were partially supported by
Fondazione di Sardegna, progetto biennale bando 2021, ``Computational Methods
and Networks in Civil Engineering (COMANCHE)'', and by the INdAM-GNCS 2022
projects ``Computational methods for kernel-based approximation and its
applications'' (LF) and ``Metodi e modelli di regolarizzazione per problemi
malposti di larga scala'' (GR).
The research of M. M.
Spalevi\'c is supported in part by the Serbian Ministry of  Science,
Technological Development, and Innovations, according to Contract
451-03-47/2023-01/200105  dated on 3 February, 2023.
The research also has been accomplished within RITA (Research Italian network
on Approximation).


\end{document}